\documentclass[11pt]{amsart}
\usepackage{amsmath,amssymb, mathtools,a4wide} % basic maths packages
\usepackage{amscd} % for commutative diagrams (rectangular)
\usepackage{amsthm} % for theorem environments
\usepackage{dsfont} % for mathematical fonts like mathbb{R}
\usepackage{MnSymbol} % for matrix dots
\usepackage{bm} % bold characters in maths...

\usepackage{float} % for fixing pictures where I want (with [H] in \begin{figure}[H]) but interferes with hyperref...

\usepackage[pdftex]{graphicx} % for including pictures with includegraphics
\usepackage{xcolor} % for colors like textcolor{red}
\usepackage[colorlinks,citecolor=blue,linkcolor=blue,urlcolor=blue]{hyperref} % for references in the text with colors%\usepackage[nottoc,numbib]{tocbibind} % pour afficher references dans la table des matieres MAIS fait disparaitre les sections en haut...

\usepackage[utf8]{inputenc} % to translate accents to intern latex code

\usepackage{tikz} % for pictures with tikz
\newtheorem{thm}{Theorem}
\newtheorem{conj}[thm]{Conjecture}
\newtheorem{lemma}[thm]{Lemma}
\newtheorem{coro}[thm]{Corollary}
\newtheorem{definition}[thm]{Definition}
\newtheorem{prop}[thm]{Proposition}
\newtheorem{example}[thm]{Example}
\newtheorem{Remark}[thm]{Remark}

\newtheorem{fakethm}{Theorem}

\newtheorem*{fakelemma}{Lemma}

\numberwithin{thm}{section} % theorem labeling within sections
\numberwithin{equation}{section} % equation labeling within sections
\numberwithin{figure}{section} % figure labeling within sections

\newcommand*{\R}{\mathbb{R}}
\newcommand*{\C}{\mathbb{C}}
\newcommand*{\M}{\mathbb{M}}

\renewcommand*{\S}{S}
\newcommand*{\g}{\mathfrak{g}}
\newcommand*{\h}{\mathfrak{h}}
\renewcommand*{\l}{\lambda}
\newcommand*{\mf}{\mathfrak}
\newcommand*{\mc}{\mathcal}
\newcommand*{\bb}{\mathbb}
\newcommand*{\del}{\partial}
\newcommand*{\delbar}{\bar\partial}
\newcommand*{\cotang}{T^*\mc{T}^n(S)}
\newcommand*{\T}{\mc{T}}

\DeclareMathOperator{\tr}{tr}

\DeclareMathOperator{\Hilb}{Hilb}
\DeclareMathOperator{\Ham}{Ham}

\DeclareMathOperator{\diag}{diag}

\DeclareMathOperator{\Hom}{Hom}

\DeclareMathOperator{\Span}{Span}
\DeclareMathOperator{\PSL}{PSL}
\DeclareMathOperator{\SL}{SL}
\DeclareMathOperator{\GL}{GL}

\DeclareMathOperator{\Lie}{Lie}

\begin{document}
\author{Alexander Thomas}
%\title{Geometric Approach to Hitchin Components via Higher Complex Structures}
%\date{}
\title{Higher Complex Structures and Flat Connections}
\address{Université Claude Bernard Lyon 1, Institut Camille Jordan, 21 bv Claude Bernard, 69100 Villeurbanne}
\email{athomas@math.univ-lyon1.fr}

\begin{abstract}
In the physics literature, Bilal--Fock--Kogan \cite{BFK} introduced the idea of parabolic reduced flat connections on a surface to give a geometric origin to $W$-algebras. In this paper, we combine these ideas with higher complex structures, geometric structures defined by Fock and the author in \cite{FockThomas}. A semiclassical analysis of the parabolic reduction establishes a direct link between flat connections and higher complex structures.

In particular, we study a certain class of connections on a bundle equipped with a line subbundle $L$, which we call $L$-parabolic. The curvature of these connections is of rank at most 1. We describe a certain family of $L$-parabolic connections with vanishing curvature, giving the data of a higher complex structure and a cotangent variation. 
%The family of connections being flat implies the compatibility condition of the cotangent variation. 
Infinitesimal higher diffeomorphisms, the natural class of transformations on higher complex structures, are realized by the infinitesimal gauge transformation induced by changing $L$. Constructing flat families of connections of this kind is linked to Toda integrable systems.
%The previous version of the paper has been completely revised.
\end{abstract}

\maketitle

\setcounter{tocdepth}{1}
\tableofcontents

\section{Introduction}

In \cite{Hit.1}, Nigel Hitchin describes, for a Riemann surface $\Sigma$ with underlying smooth surface $S$, a connected component of the \textbf{character variety}
$$\mathcal{X}(S, \PSL_n(\R)) := \Hom(\pi_1 S, \PSL_n(\R))/\PSL_n(\R),$$
which he parametrizes by holomorphic differentials. These components are called \textbf{Hitchin components}. 
His approach is complex analytic and uses the non-abelian Hodge correspondance \cite{Corlette, Donaldson, Simpson, Hit.2}, which in turn is closely linked to the hyperkähler structure of the moduli space of polystable Higgs bundles. Elements of these components can also be described as representations $\pi_1(S)\to \PSL_n(\R)$ which can be continously deformed to a representation factoring through a fuchsian representation $\pi_1(S)\to\PSL_2(\R)$. All these representations are discrete and faithful \cite{Lab,FG}.

For $\PSL_2(\R)$, the Hitchin component is the \textbf{Teichmüller space}, which is the moduli space of various geometric structures on the underlying smooth surface $\S$, for example complex structures or hyperbolic structures, considered up to isotopies. Thus, the question naturally arises \textit{whether there is a geometric structure on $\S$ whose moduli space gives Hitchin's component for higher rank}. This question already appears in Hitchin's original paper \cite{Hit.1}.

A vast body of literature is dedicated to this question. To mention only a few, there is the work of Choi--Goldman \cite{CG} on convex projective structures, the work of Guichard--Wienhard \cite{GW} on special projective structures, and the crucial notion of Anosov representations introduced by Labourie \cite{Lab}.

\medskip
In \cite{FockThomas} another candidate for such a geometric structure is constructed, called the \textbf{higher complex structure}, a generalization of the complex structure. A higher complex structure of order $n$ can be thought of as an $(n-1)$-jet of a complex curve inside $T^{*\C}S$ along the zero-section. Taking the underlying 1-jet induces a complex structure on $S$.

Higher complex structures are considered up to the action of Hamiltonian diffeomorphisms of $T^*S$ preserving the zero-section, called \textbf{higher diffeomorphisms}. The moduli space $\T^n(S)$ of higher complex structures of order $n$ shares multiple properties with the Hitchin component, in particular both are contractible and of same dimension \cite{FockThomas, Nolte}.

The main open question about higher complex structures is whether there is a canonical diffeomorphism between $\T^n(S)$ and the $\mathrm{PSL}_n(\R)$-Hitchin component. More generally, one might expect a canonical diffeomorphism between a tubular neighborhood of $\T^n(S)\subset T^*\T^n(S)$ and a tubular neighborhood of the Hitchin component inside the complex character variety $\mathcal{X}(S,\PSL_n(\C))$. This would generalize work of Donaldson \cite{Donaldson}, surveyed by Trautwein \cite{Trautwein}, on almost-Fuchsian representations in the case $n=2$.

\medskip
Our main inspiration for this paper comes from ideas which have been described by Bilal--Fock--Kogan in \cite{BFK}, treating about the geometric origins of $W$-algebras. There, a certain class of flat connections is considered, obtained as the Hamiltonian reduction with respect to a parabolic subgroup of the gauge group. This reduction makes appear a special class of tensors, encoding some sort of generalized complex and projective structures. These tensors fulfil a compatibility condition, coming from the flatness of the connection.

Higher complex structures and their cotangent variations can be seen as the mathematical rigorous formulation of these geometric structures. The geometric description of the transformations of the tensors associated to generalized complex and projective structures, which remained unsettled in \cite{BFK}, is given in terms of higher diffeormophisms.

In this paper, we use a semiclassical limit to deduce higher complex structures and their cotangent variations from a special flat family of connections and find a gauge-theoretic description of the infinitesimal higher diffeomorphism action.

\bigskip
\noindent \textbf{\textit{Structure and results.}}
In this paper, we study the links between higher complex structures together with a cotangent vector (i.e. a point in $T^*\T^n(S)$) and flat connections.

More precisely, we show that a higher complex structure on $S$ induces a complex bundle $V$ together with a holomorphic line subbundle $L$ isomorphic to $K^{(n-1)/2}$ (where we use the induced complex structure on $S$ and $K$ denotes the canonical bundle), and a special field $\Phi\in\Omega^1(S,\mathrm{End}(V))$, see Section \ref{indbundle}.

We then study in Section \ref{parabolicreduction} the symplectic reduction of the space of all connections on $V$ by the group of gauge transformations preserving $L$. A representative of such connections is called $\mathbf{L}$\textbf{-parabolic}. An $L$-parabolic connection is almost flat (its curvature is of rank at most 1) and generically allows a standard representative in its gauge class, called the $\mathbf{L}$\textbf{-parabolic gauge}. We give local coordinates of the moduli space of $L$-parabolic connections, see Proposition \ref{Prop:para-L-para-conns}. Changing $L$ induces a change of the $L$-parabolic gauge which we compute explicitly.

In the core of the paper, Section \ref{parabolicwithlambda}, we consider a family of $L$-parabolic connections on $V$ with parameter $\l\in\C^*$ of the form 
\begin{equation}\label{Eq:C-intro}
C(\l)=\l\Phi+d_A+\l^{-1}\Psi,
\end{equation}
where $\Phi$ comes from a higher complex structure, $d_A$ is a connection satisfying $d_A(\Phi)=0$, and $\Psi\in\Omega^1(S,\mathrm{End}(V))$. An interesting example is when $C(\l)$ is flat for all $\l$.
We show existence of the $L$-parabolic gauge, in which we can locally parametrize $C(\l)$ by functions $(\bm\hat{\mu}_k(\l),\bm\hat{t}_k(\l))_{2\leq k\leq n}$. The highest order terms of these parameters, in the Taylor expansion for $\l\to\infty$, are tensors $(\mu_k,t_k)_{2\leq k\leq n}$. The first main result of the paper shows that these tensors encode the data of a higher complex structure with a cotangent variation. The $\mu_k$ describe the higher complex structure, whereas the $t_k$ describe the covector. In $T^*\T^n(S)$, the covector has to satisfy a certain compatibility condition. We say that the covector is $\boldsymbol{\mu}$\textbf{-holomorphic}. 

\begin{fakethm}[Theorem \ref{conditioncinconnection}]
If the family $C(\l)$ from Equation \eqref{Eq:C-intro} is flat, then the covector $(t_2,...,t_n)$ is $\mu$-holomorphic.
\end{fakethm} 

The proof of the theorem relies on semiclassical analysis when $\l$ goes to infinity. When $C(\l)$ is flat, the functions $(\bm\hat{\mu}_k(\l),\bm\hat{t}_k(\l))_{2\leq k\leq n}$ satisfy a compatibility constraint. The highest order term in the Taylor expansion of this compatibility constraint turns out to be the condition for the covector to be $\mu$-holomorphic.

The $L$-parabolic gauge we used to represent $C(\l)$ can be altered by changing the line subbundle $L$. The second main result shows that this induces a gauge transformation which reproduces the infinitesimal higher diffeomorphism action on the tensors $(\mu_k,t_k)_{2\leq k\leq n}$. This gives a first hint towards a gauge-theoretic implementation of the higher diffeomorphism action. Another viewpoint, also at the infinitesimal level, is given in \cite[Thm F]{Fock-bundles}.

\begin{fakethm}[Theorem \ref{actionsymponlambdaconn}]
Let $C(\l)$ be a family of flat connections as above. Then, the infinitesimal gauge action, induced by changing the line subbundle $L$, on the highest terms $(\mu_k,t_k)_{2\leq k\leq n}$ of the parameters $(\bm\hat{\mu}_k(\l),\bm\hat{t}_k(\l))_{2\leq k\leq n}$ of $C(\l)$, is the same as the infinitesimal higher diffeomorphism action on the higher complex structure with cotangent vector described by $(\mu_k,t_k)_{2\leq k\leq n}$.
\end{fakethm}

The proof relies on the fact that the semiclassical limit of the bracket of differential operators is the Poisson bracket of the associated symbols.

In the last part, Section \ref{finalstep}, we analyze how to construct $L$-parabolic connections $C(\l)$ with some reality constraint from higher complex structures and a cotangent vector. The reality constraint implies that the family $C(\l)$ is flat (see Proposition \ref{Prop:flat-para-equiv}).
In the special case where the higher complex structure is trivial and the covector is zero, the flatness equations reduce to the $\mathfrak{sl}_n$-Toda integrable system. In the general case, we argue that the flatness equations can be seen as a generalized Toda system, using the loop algebra $\mathcal{L}(\mathfrak{sl}_n)=\mathfrak{sl}_n\otimes \C[\l,\l^{-1}]$. For $\mathfrak{sl}_2$, this generalized Toda system reduces to the cosh-Gordon equation, studied in \cite{Fock}.

In Appendix \ref{appendix:C}, we prove a technical lemma from the main text.

\bigskip
\noindent \textbf{\textit{Motivation.}}
We describe a conjectural broader picture, summarized in Figure \ref{HK}, which greatly motivated the paper.

Hitchin's original approach in \cite{Hit.1} to construct connected components in real character varieties uses the hyperkähler structure of the moduli space of Higgs bundles $\mc{M}_H(\Sigma)$. The left hand side of Figure \ref{HK} shows the twistor space of $\mathcal{M}_H(\Sigma)$ (a 1-parameter family of Kähler manifolds parametrized by $\C P^1$), together with the Hitchin fibration. The general theory of twistor spaces shows that all fibers of the twistor space are diffeomorphic \cite{HKLR}. For $\mathcal{M}_H(\Sigma)$, this gives the \emph{non-abelian Hodge correspondence} \cite{Corlette, Donaldson2, Hit.2, Simpson}.

\begin{figure}[h!]
\centering
\begin{tikzpicture}[scale=1.5]
	\draw (0,0) circle (1cm);
	%\draw [domain=0:1] plot (\x, {\x /3}) ; 
\draw [domain=180:360] plot ({cos(\x)},{sin(\x)/3});
	\draw [domain=0:180, dotted] plot ({cos(\x)},{sin(\x)/3});
	\draw [fill=white] (0,1) circle (0.04);
	\draw [fill=white] (0,-1) circle (0.04);
	\draw (0,1.2) node {$\overline{\mc{M}}_H(\Sigma)$};
	\draw[below] (0,1) node {$0$};
	\draw (0,-1.2) node {$\mc{M}_H(\Sigma)$};
	\draw[above] (0,-1) node {$\infty$};
	\draw[->] (0,-1.4)--(0,-1.75);
	\draw (0,-2) node {$\bigoplus_{i=2}^n \mathrm{H}^0(\Sigma,K^i)$};
	\draw [domain=-28:32, dashed, ->] plot ({0.33*cos(\x)}, {0.33*sin(\x)-1.5});
	
	\draw (1,-1.45) node {Hitchin};
	\draw (1,-1.67) node {section};
	\draw (-1.25,-1.45) node {Hitchin};
	\draw (-1.25,-1.7) node {fibration};
	\draw (1.85,0) node {$\mathcal{X}(\S, \mathrm{SL}_n(\C))$};
	%\draw (1.65,-0.18) node {$\cong \mc{A}//\mc{G}$};
	
	\begin{scope}[xshift=4.6cm]
	\draw (0,0) circle (1cm);
	%\draw [domain=0:1] plot (\x, {\x /3}) ; 
	\draw [domain=180:360] plot ({cos(\x)},{sin(\x)/3});
	\draw [domain=0:180, dotted] plot ({cos(\x)},{sin(\x)/3});
	\draw [fill=white] (0,1) circle (0.04);
	\draw [fill=white] (0,-1) circle (0.04);
	\draw (0,1.2) node {$\overline{U}$};
	\draw[below] (0,1) node {$0$};
	\draw (0,-1.2) node {$U \subset \cotang$};
	\draw[above] (0,-1) node {$\infty$};
	\draw[->] (0,-1.4)--(0,-1.75);
	\draw (0,-2) node {$\T^n(S)$};
	\draw [domain=-40:24, dashed, ->] plot ({0.33*cos(\x)}, {0.33*sin(\x)-1.5});
	
	\draw (1,-1.45) node {zero-};
	\draw (1,-1.67) node {section};
	\draw (-1.25,-1.45) node {canonical};
	\draw (-1.25,-1.7) node {projection};
	\draw (2.14,0) node {$U'\subset\mathcal{X}(\S, \mathrm{SL}_n(\C))$};
	%\draw (2,-0.18) node {$(\mc{A}//\mc{P})//\Ham_0$};
	\end{scope}
	\end{tikzpicture}
	\caption{Twistor space for moduli space of Higgs bundles and $U\subset \cotang$}\label{HK}
\end{figure}

Optimistically, there is a similar twistor space for a tubular neighborhood $U$ of $\T^n(S)\subset\cotang$, diffeomorphic to a neighborhood $U'$ of the Hitchin component inside the complex character variety. The role of the Hitchin fibration is simply the projection map onto $\T^n(S)$.

\bigskip
\noindent \textbf{\textit{Notations.}} Throughout the paper, $S$ denotes a smooth oriented closed surface of genus $g \geq 2$. When $S$ is equipped with a complex structure, we call the associated Riemann surface $\Sigma$. A complex local coordinate system on $\Sigma$ is denoted by $(z, \bar{z})$. Linear coordinates on $T^{*\mathbb{C}}\Sigma$ are denoted by $(p,\bar{p})$. The canonical bundle is $K=T^{*(1,0)}\Sigma$ and we use the shorthand notation $K^i=K^{\otimes i}$. The space of sections of a bundle $B$ is denoted by $\Gamma (B)$. The Hamiltonian reduction (or symplectic reduction, or Marsden-Weinstein quotient) of a symplectic manifold $X$ by a group $G$ is denoted by $X//G$, where the reduction is over the zero-coadjoint orbit. The equivalence class of some element $a$ is denoted by $[a]$.

\bigskip
\noindent \textbf{\textit{Acknowledgments.}}
I warmly thank Vladimir Fock, Georgios Kydonakis, Alex Nolte and Charlie Reid for fruitful discussions and advice. I also thank the reviewers for many helpful comments. A previous version of the paper is part of my PhD thesis accomplished at the University of Strasbourg. For the new version, I was supported by the Deutsche Forschungsgemeinschaft (Project-ID 281071066 - TRR 191). The new version is a complete rewriting, clarifying and rectifying statements and including more general results.

\section{Higher complex structures}\label{highercomplex}

The goal of this section is twofold: first we give a summary of \cite{FockThomas}, in particular the construction of higher complex structures, their moduli space and the cotangent bundle to its moduli space. 
Second, we present some new material, especially a bundle with several extra structures induced by the higher complex structure. This bundle is crucial in the subsequent sections.

\subsection{Higher complex structures}

Given a smooth surface $S$ equipped with some reference complex structure, another complex structure on $S$ can be described by the \textbf{Beltrami differential} $\mu\in\Gamma(K^{-1}\otimes \bar{K})$ where $K$ denotes the canonical bundle of the reference complex structure. It determines the notion of a local holomorphic function $f$ by the condition $(\delbar-\mu\del)f=0$, where we use $\partial=\tfrac{\partial}{\partial z}$ and $\bar\partial=\tfrac{\partial}{\partial \bar{z}}$. The Beltrami differential determines a linear direction in $T^{\C}\S$, the direction generated by the vector $\delbar-\mu\del$. Since $\delbar-\mu\del$ and $\del-\bar{\mu}\delbar$ have to be linearly independent, we get the condition $\mu\bar{\mu}\neq 1$.
Equivalently, we can use the cotangent bundle $T^*\S$, and say that the complex structure is \emph{entirely encoded in a nowhere real section of $\bb{P}(T^{*\C}\S)$}.

The idea of higher complex structures is to replace the linear direction by a polynomial direction, or more precisely an $n$-jet of a curve inside $T^{*\C}\S$ along the zero-section.
To get a precise definition, we use the \textbf{punctual Hilbert scheme} of the plane, denoted by $\Hilb^n(\C^2)$ which is defined by
$$\Hilb^n(\C^2)=\{I \text{ ideal of } \C[x,y] \mid \dim \C[x,y]/I = n\}.$$
A generic point in $\Hilb^n(\C^2)$ is an ideal whose associated algebraic variety is a collection of $n$ distinct points in $\C^2$. A generic ideal can be written as $$\langle-x^n+t_1x^{n-1}+...+t_n, -y+\mu_1+\mu_2x+...+\mu_nx^{n-1} \rangle.$$ Moving around in $\Hilb^n(\C^2)$ corresponds to a movement of $n$ particles in $\C^2$. But whenever $k$ particles collide the Hilbert scheme retains an extra information: the $(k-1)$-jet of the curve along which the points entered into collision. The \textbf{zero-fiber}, denoted by $\Hilb^n_0(\C^2)$, consists of those ideals whose support is the origin. A generic point in $\Hilb^n_0(\C^2)$ is of the form $$\langle x^n, -y+\mu_2x+\mu_3x^2+...+\mu_nx^{n-1}\rangle$$
which can be interpreted as an $(n-1)$-jet of a curve at the origin. For more details on the punctual Hilbert scheme, we refer to \cite[Appendix A]{g-complex}.

We can now give the definition of a higher complex structure:

\begin{definition}[Def.2 in \cite{FockThomas}]
 A \textbf{higher complex structure} of order $n$ on a surface $S$ is a section $I$ of $\Hilb^n_0(T^{*\mathbb{C}}S)$ such that at each point $z\in \S$ we have $I(z)+\bar{I}(z)=\langle p, \bar{p} \rangle$, the maximal ideal supported at the origin of $T_z^{*\mathbb{C}}S$. 
The space of all higher complex structures of order $n$ on $S$ is denoted by $\M^n(S)$.
\end{definition}

Notice that we apply the punctual Hilbert scheme pointwise, giving a  bundle over $S$ with fiber $\Hilb^n_0(T_z^{*\C}S)$ over $z\in S$.
The condition on $I+\bar{I}$ ensures that $I$ is a generic ideal, so locally it can be written as
\begin{equation}\label{I-expr}
I(z)=\langle p^n, -\bar{p}+\mu_2(z, \bar{z})p+\mu_3(z, \bar{z}) p^2...+\mu_n(z, \bar{z})p^{n-1}\rangle.
\end{equation}
Here $(p,\bar{p})$ denote linear coordinates on $T^{*\C}S$, which can be identified with $(\partial,\bar\partial)$ (seen as elements in the bidual). The coefficients $\mu_k$ are called \textbf{higher Beltrami differentials}. A direct computation shows that $\mu_k$ is a tensor of type $(1-k,1)$, i.e. $\mu_k \in \Gamma(K^{1-k}\otimes \bar{K})$. The coefficient $\mu_2$ is the usual Beltrami differential. In particular for $n=2$ we get the usual complex structure. Forgetting all $\mu_k$ with $k\geq 3$, we get:
\begin{prop}
A higher complex structure induces a complex structure on $S$.
\end{prop}
We will often use this induced complex structure and consider the associated Riemann surface $\Sigma$. The induced complex coordinates are characterized by $\mu_2(z,\bar{z})=0$.

The punctual Hilbert scheme admits an equivalent description as a space of pairs of commuting operators. To an ideal $I$ of $\C[x,y]$ of codimension $n$, one can associate the multiplication operators by $x$ and by $y$ in the quotient $\C[x,y]/I$, denoted by $M_x$ and $M_y$. This gives a pair of commuting operators. Conversely, to two commuting operators $(A,B)$ we can associate the ideal $I(A,B)=\{P\in\C[x,y] \mid P(A,B)=0\}$. 
The zero-fiber $\Hilb^n_0(\C^2)$ corresponds to nilpotent commuting operators.

From this point of view, a higher complex structure is locally a gauge class of special matrix-valued 1-forms of the form $\Phi_1dz+\Phi_2d\bar{z}$ where $(\Phi_1, \Phi_2)$ is a pair of commuting nilpotent matrices with $\Phi_1$ principal nilpotent (which means of maximal rank $n-1$).

\medskip
\noindent To define a finite-dimensional moduli space of higher complex structures, we have to define some equivalence relation. It turns out that the good notion is the following:
\begin{definition}[Def.3 in \cite{FockThomas}]
A \textbf{higher diffeomorphism} of a surface $S$ is a Hamiltonian diffeomorphism of $T^*S$ preserving the zero-section $S \subset T^*S$ setwise at all times. The group of higher diffeomorphisms is denoted by $\Ham_0(T^*S)$.
\end{definition}
Symplectomorphisms act on $T^{*\C}S$, so also on 1-forms. From this, one can heuristically expect an action of higher diffeomorphisms on a higher complex structure, considered as the limit of an $n$-tuple of 1-forms. This action has been described in \cite{Nolte} using a jet perspective.
%We then consider higher complex structures modulo higher diffeomorphisms, i.e. two structures are equivalent if one can be obtained by the other by applying a higher diffeomorphism.
%Locally, all higher complex structures are equivalent:
%\begin{thm}[Theorem 1 in \cite{FockThomas}]\label{loctriii}
%The higher complex structure can be locally trivialized, i.e. there is a higher diffeomorphism which sends the structure to $(\mu_2(z,\bar{z}),...,\mu_n(z,\bar{z}))=(0,...,0)$ for all small $z\in \C$.
%\end{thm}
We define the \textbf{moduli space of higher complex structures}, denoted by $\T^n(S)$, as the space of higher complex structures $\mathbb{M}^n(S)$ modulo higher diffeomorphisms.
The main properties are given in the following theorem:
\begin{thm}[Theorem 2 in \cite{FockThomas}, Theorem 1.1 in \cite{Nolte}]\label{mainresultncomplex}
For a surface $\S$ of genus $g\geq 2$ the moduli space $\T^n(\S)=\M^n(\S)/\Ham_0(T^*\S)$ is a contractible manifold of complex dimension $(n^2-1)(g-1)$. The cotangent space at any point $\mu=[(\mu_2,...,\mu_n)]$ is given by 
$$T^*_{\mu}\mathcal{T}^n(S) = \bigoplus_{m=2}^{n} \mathrm{H}^0(\Sigma,K^m),$$
where $\Sigma$ is the Riemann surface given by the point in Teichmüller space. Finally, there is a copy of Teichmüller space $\mc{T}^2(\S)\rightarrow \T^n(\S)$ and a forgetful map $\mathcal{T}^n(\S) \rightarrow \mathcal{T}^{n-1}(\S)$.
\end{thm}
The forgetful map in coordinates is just given by forgetting the last Beltrami differential $\mu_n$. The copy of Teichmüller space is given by $\mu_3=...=\mu_n=0$ (this relation is unchanged under higher diffeomorphisms).

%We notice the similarity to Hitchin's component, especially the contractibility, the dimension and the copy of Teichmüller space inside (the Fuchsian locus). For these reasons, we believe in \emph{the existence of a canonical diffeomorphism between $\mathcal{T}^n(S)$ and the $\PSL_n(\R)$-Hitchin component}. We ask in particular for the diffeomorphism to be equivariant with respect to the mapping class group action. 
%At the end of the paper in Section \ref{finalstep} we indicate how to link $\T^n(\S)$ to Hitchin's component. Assuming a strong conjecture (an analog of the non-abelian Hodge correspondence in our setting), we prove that $\T^n(\S)$ is canonically diffeomorphic to Hitchin's component in Theorem \ref{mainthmm}.

\subsection{Cotangent bundle of higher complex structures}\label{cotangs}

In our subsequent study of flat connections, higher complex structures will appear together with a cotangent vector satisfying a compatibility condition, $\mu$-holomorphicity. This can be best understood by studying the total cotangent bundle $\cotang$. We present here heuristic arguments, allowing the reader to get the broad idea. The detailed proofs use different arguments and can be found in \cite{FockThomas, Nolte}.

The punctual Hilbert scheme inherits a complex symplectic structure from $\C^2$. The zero-fiber $\Hilb^n_0(\C^2)$ is Lagrangian in $\Hilb^n_{red}(\C^2)$, the \textbf{reduced Hilbert scheme}, which consists by definition of those ideals $I$ whose support has barycenter the origin (generically $n$ points with barycenter equal to the origin).
Using the symplectic form, we then get $$T^*\Hilb^n_0(\C^2)\cong T^{\perp}\Hilb^n_{0}(\C^2) \approx \Hilb^n_{red}(\C^2),$$ where $T^{\perp}$ denotes the normal bundle. Near the zero-section, to first order, the normal bundle can be identified with the whole space, here the reduced Hilbert scheme.

The cotangent bundle to a quotient space $X/G$, where $X$ is a manifold and $G$ a Lie group, is the Hamiltonian reduction: $T^*(X/G) \cong T^*X//G$. Using this, we can heuristically compute
\begin{align}
\cotang &= T^*\left(\Gamma(\Hilb^n_0(T^{*\mathbb{C}}S))/\Ham_0(T^*S)\right)  \nonumber \\
&= \Gamma(T^*\Hilb^n_0(T^{*\mathbb{C}}S)) // \Ham_0(T^*S)  \nonumber \\
&= \Gamma(T^{normal}\Hilb^n_0(T^{*\mathbb{C}}S))// \Ham_0(T^*S)  \nonumber \\
&= \Gamma(\Hilb^n_{red}(T^{*\mathbb{C}}S)) // \Ham_0(T^*S) \; \mod t^2. \label{hkquotientofmodulispace}
\end{align}

\begin{Remark}
We stress again that the argument above is only heuristic. In \cite{Nolte}, Nolte shows that the higher diffeomorphism action is not proper. Hence the symplectic reduction cannot be carried out by the standard techniques.
\end{Remark}

An element of $\Gamma(\Hilb^n_{red}(T^{*\C}\S))$ is an ideal locally of the form 
\begin{equation}\label{ideal-I}
I=\langle p^n-t_2p^{n-2}-...-t_n, -\bar{p}+\mu_1+\mu_2p+...+\mu_np^{n-1}\rangle,
\end{equation}
where $(t_k, \mu_k)_{2\leq k \leq n}$ are complex-valued functions, which can serve as coordinates. Globally, the $t_k$ are tensors of type $(k,0)$.
The coefficient $\mu_1$ is an explicit function of the other variables using that the barycenter is the origin. The fact that the normal bundle is the total space to order 1 is expressed by ``mod $t^2$'', meaning that all quadratic or higher terms in the $t_k$ have to be dropped.

To compute the moment map, we have to understand with more detail the action of higher diffeomorphisms on $\Gamma(\Hilb^n_{red}(T^{*\mathbb{C}}S))$. The ideal $I$ has two generators which we put into the form $p^n-P(p)$ and $-\bar{p}+Q(p)$ where $P(p)=t_2p^{n-2}+...+t_n$ and $Q(p)=\mu_2p+...+\mu_np^{n-1}$. A higher diffeomorphism generated by some Hamiltonian $H$ acts on $I$ by changing the two polynomials. Their infinitesimal variations $\delta P$ and $\delta Q$ are given by (see \cite{FockThomas} after Proposition 2)
\begin{align}
\delta P &= \{H, p^n-P(p)\} \mod I \nonumber \\
\delta Q &= \{H, -\bar{p}+Q(p)\} \mod I. \label{idealvariation}
\end{align}

\begin{Remark}
One can easily show that only the class $H \mod I$ acts, which allows to restrict attention to $H(p,z,\bar{z})$, polynomial in $p$ of degree at most $n-1$.
\end{Remark}

Using these variation formulas, one can compute the moment map:
\begin{thm}[Theorem 3 in \cite{FockThomas}]\label{conditionC}
The cotangent bundle to the moduli space of higher complex structures $\cotang$ is given by $\Ham_0(T^*\S)$-equivalence classes $[(\mu_k,t_k)]_{2\leq k\leq n}$, where for all $k=2,...,n$ we have $\mu_k \in \Gamma(K^{1-k}\otimes \bar{K})$ and $t_k \in \Gamma(K^k)$ satisfying
\begin{equation}\label{Eq:mu-holomorphic}
(-\bar{\partial}\!+\!\mu_2\partial\!+\!k\partial\mu_2)t_{k}+\sum_{\ell=1}^{n-k}((\ell\!+\!k)\partial\mu_{\ell+2}+(\ell\!+\!1)\mu_{\ell+2}\partial)t_{k+\ell}=0.
\end{equation}
\end{thm}
We call the moment map \boldsymbol{$\mu$}\textbf{-holomorphicity term} and say that the covector $(t_2,...,t_n)$ is $\boldsymbol{\mu}$\textbf{-holomorphic}. It generalizes the usual notion of being holomorphic $\bar\partial t_k=0$, to which it reduces in the special case when $\mu_\ell=0$ for all $\ell$.

\subsection{Induced bundle}\label{indbundle}

To any section $I$ of $\Hilb^n(T^{*\C}\S)$, we can canonically associate a vector bundle $V'$ of rank $n$ over $\S$ whose fiber over a point $z\in S$ is $\mathbb{C}[p,\bar{p}]/I(z)$. Consider $\mathrm{Sym}(T^{\C}\S)$, the space of functions on $T^{*\C}\S$ which are polynomial in each fiber. The bundle $V'$ is the quotient of $\mathrm{Sym}(T^{\C}\S)$ by $I$. The image of the constant functions under the canonical projection $\mathrm{Sym}(T^{\C}\S)\to V'$ gives a trivial complex line subbundle $L'\subset V'$.

We restrict now to the case where $I$ is a higher complex structure. We then get a complex structure on $S$ with local complex coordinates $(z,\bar{z})$. 
%From the matrix viewpoint of the punctual Hilbert scheme, we also get a volume form on each fiber of $V'$ (this is true more generally for $I$ a section of $\Hilb^n_{red}(T^{*\C}S)$). 
Finally, we get a field $\Phi \in \Gamma(\mathfrak{gl}(V')\otimes T^{*\C}\S)$, i.e. an $\mf{gl}_n$-valued 1-form locally of the form $\Phi = \Phi_1dz+\Phi_2d\bar{z}$, where $(\Phi_1,\Phi_2)$ is a pair of commuting nilpotent matrices. $\Phi_1$ acts on $V'$ as multiplication by $p$ and $\Phi_2$ as multiplication by $\bar{p}$. At a point $z\in\S$, any non-zero vector of $L'(z)$ is a cyclic vector for $\Phi_1$.

Consider a nowhere vanishing section $s$ of $L'$, which exists since $L'$ is trivial. On a local chart, we get a basis $B=(s,\Phi_1 s,\Phi_1^2s,...,\Phi_1^{n-1}s)$, applying successively $\Phi_1$ to $s$. We can also write $B$ as $(s,ps,p^2s,...,p^{n-1}s)$ since $\Phi_1$ is multiplication by $p$. In this basis, we recover the local structure of the higher complex structure $I$:
$$I=\langle p^n, -\bar{p}+\mu_2 p+\mu_3p^2+...+\mu_np^{n-1} \rangle.$$
Using the local basis $B$, we can give $V'$ a holomorphic structure. Indeed, under a holomorphic change of coordinates $z\mapsto w(z)$, the basis element $p^ks$ transforms into $(\frac{dw}{dz})^kp^ks$ which serves as holomorphic transition map. As holomorphic bundle, we have $V'\cong\mathcal{O}\oplus K^{-1}\oplus K^{-2}\oplus...\oplus K^{-(n-1)}$, where $\mathcal{O}$ denotes the trivial line bundle.
To get a vector bundle of degree zero, we fix a square root of the canonical bundle $K^{1/2}$ (i.e. a spin structure on $\S$) and we define $$V := V'\otimes K^{(n-1)/2} = K^{(n-1)/2}\oplus K^{(n-3)/2}\oplus...\oplus K^{-(n-1)/2}.$$
We also define $L:=L'\otimes K^{(n-1)/2}$. The 1-form $\Phi$ still acts on $V$, acting trivially on $K^{(n-1)/2}$. Now we can globally write $\Phi^{1,0}$ as principal nilpotent matrix consisting of one Jordan block and $\Phi^{0,1}$ as given by the higher Beltrami differentials:

\begin{equation}\label{Eq:matrix-Phi}
\Phi^{1,0} = \begin{pmatrix}   0&&&&  \\ 1&0  &&& \\ &1 &\ddots  && \\ &&\ddots &0& \\ &&&1&0\end{pmatrix} \;\text{ and } \;\; \Phi^{0,1} = \begin{pmatrix}   0&&&&  \\ \mu_2 &0  &&& \\ \mu_3& \mu_2 &\ddots  && \\ \vdots & \ddots &\ddots &0& \\ \mu_n&\cdots&\mu_3&\mu_2&0\end{pmatrix}.
\end{equation}
Here, empty entries are all zero and the constants 1 have to be interpreted as the canonical section of $\Hom(K^{(n-i)/2},K^{(n-i-2)/2}\otimes K) \cong \mathcal{O}$. We also recover $\mu_k\in \Gamma(K^{1-k}\otimes \bar{K})$.

Finally, the higher complex structure $I$ induces a filtration on $V$, i.e. a complete flag in each fiber. Put $I_k=I+\langle p,\bar{p} \rangle^k$, so we have $I_0=\C[p, \bar{p}] \supset I_1 \supset ...\supset I_n = I$. Define $$F_{n-k}=\ker (\mathbb{C}[p,\bar{p}]/I \rightarrow \C[p, \bar{p}]/I_k).$$ Then, $F_1\subset F_2\subset ...\subset F_n$ form an increasing complete flag with $\dim F_k=k$. Equivalently, we have $F_k(z)=\ker \Phi(v)^k$, where $v$ is any non-zero real tangent vector $v\in T_z\S$. From the definition of a higher complex structure, one can easily check that this is independent of the vector field $v$. In the local basis $B$, we have $F_k=\Span(p^{n-k}s,...,p^{n-1}s)$.

\begin{prop}
The filtration on $V$ induced by a higher complex structure is preserved under higher diffeomorphisms.
\end{prop}
\begin{proof}
We show that the filtration is preserved under the infinitesimal action of higher diffeomorphisms. By integration, the result then follows.

Higher diffeomorphisms act on all objects induced from higher complex structures, in particular on $V$ and the filtration.
Preserving the filtration is a local property, so we can work in a basis of the form $B=(s,ps,...,p^{n-1}s)$. Further take a Hamiltonian $H\in \mathcal{C}_{pf}(T^{*\C}\S)$ generating a higher diffeomorphism, in particular vanishing on the zero-section $\S\subset T^*\S$. A basis element $p^ks$ changes by $$\{H, p^ks\}=p^k\{H,s\}+kp^{k-1}(\del H)s.$$ Both terms lie in $F_{n-k}$ since $H$, considered as polynomial in $p$, has no constant term. So the action of $\Ham_0(T^*S)$ on $V$ is lower triangular. Hence, the flag structure is preserved. 
\end{proof}

\begin{Remark}
The integration of the infinitesimal higher diffeomorphism action has been done in \cite{Nolte}, using jet bundles. In particular, the filtration on $V$ corresponds to the filtration by degree in the bundle of jets, which is clearly preserved under higher diffeomorphisms.
\end{Remark}

%The link between higher complex structures and character varieties relies on the idea to deform the 1-form $\Phi$ to a flat connection. 

\subsection{Conjugated higher complex structures}\label{dualcomplexstructure}

There is a natural notion of conjugated space to $\cotang$ using the natural complex conjugation on the complexified cotangent bundle $T^{*\C}S$. To an ideal $I\in \Hilb^n_{red}(T^{*\C}S)$, we associate the ideal $\bar{I}$ and then consider its $\Ham_0(T^*S)$-equivalence class.

In coordinates, we start from $$I=\langle p^n-t_2p^{n-2}-...-t_n, -\bar{p}+\mu_1+\mu_2p+...+\mu_np^{n-1} \rangle.$$
To get the conjugated structure, we have to express $\bar{I}$ in the same form as $I$, i.e. as 
\begin{align*}
\bar{I} &=\langle \bar{p}^n-\bar{t}_2\bar{p}^{n-2}-...-\bar{t}_n, -p+\bar{\mu}_1+\bar{\mu}_2\bar{p}+...+\bar{\mu}_n\bar{p}^{n-1} \rangle \\
&= \langle p^n-{}_2tp^{n-2}-...-{}_nt, -\bar{p}+{}_1\mu+{}_2\mu p+...+{}_n\mu p^{n-1} \rangle.
\end{align*}
where $({}_kt, {}_k\mu)$ are the parameters of the conjugate to $\cotang$.
It is possible to explicitly express the conjugated coordinates $({}_kt, {}_k\mu)_{2\leq k\leq n}$ in terms of $(t_k, \mu_k)_{2\leq k\leq n}$.
For example one finds ${}_2\mu=\frac{1}{\bar{\mu}_2}$ and ${}_nt = \bar{\mu}_2^n\bar{t}_n$.

\section{$L$-parabolic connections and reduction}\label{parabolicreduction}

In this section, we describe a special class of connections, called $L$-parabolic. They appear in a vector bundle $V$ with a given line subbundle $L$. The space of $L$-parabolic connections can be obtained by a Hamiltonian reduction. A generic $L$-parabolic connection has a canonical standard form, called $L$-parabolic gauge. We give local coordinates of $L$-parabolic connections and study the gauge transformation induced by changing $L$. 

The idea of the parabolic gauge comes from work of Bilal--Fock--Kogan \cite{BFK} in the work on $W$-algebras in conformal field theories.
%In that paper, the authors also describe some ideas for generalized complex and projective structures. Our higher complex structures are the mathematically rigorous version of their ideas. 
Our treatment of the parabolic reduction follows different notation. The main new ingredient is the use of a semiclassical limit, which gives a link to the theory of higher complex structures, see Section \ref{parabolicwithlambda}.

\subsection{Atiyah-Bott reduction}

Before going to the $L$-parabolic reduction, we recall the classical reduction of connections by gauge transforms, developed by Atiyah and Bott in their famous paper \cite{AtBott}.

Let $S$ be a surface and $G$ be a semisimple Lie group with Lie algebra $\g$. Let $E$ be a principal $G$-bundle over $S$. Denote by $\mathcal{A}$ the space of all $\g$-connections on $E$. It is an affine space modeled over the vector space of $\g$-valued 1-forms $\Omega^1(S, \g)$. Further, denote by $\mathcal{G}$ the space of all gauge transforms, i.e. bundle automorphisms. We can identify the gauge group with $G$-valued functions: $\mathcal{G}=\Omega^0(S,G)$.

On the space of all connections $\mathcal{A}$, there is a natural symplectic structure given by $$\bm\hat{\omega} = \int_{S} \tr \;\delta A \wedge \delta A,$$ where tr
denotes the Killing form on $\g$ (the trace for matrix Lie algebras).
Since $\mathcal{A}$ is an affine space, its tangent space at every point is canonically isomorphic to $\Omega^1(S, \g)$. So given $d_A \in \mathcal{A}$ and $A_1, A_2 \in T_{d_A}\mathcal{A} \cong \Omega^1(S, \g)$, we have $\bm\hat{\omega}_{d_A}(A_1, A_2) = \int_{S} \tr \; A_1\wedge A_2$. Note that $\bm\hat{\omega}$ is constant (independent of $d_A$) so $d\bm\hat{\omega} = 0$. Further, the 2-form $\bm\hat{\omega}$ is clearly antisymmetric and non-degenerate (since the Killing form is). Remark finally that this construction only works on a surface since we integrate a 2-form.

We can now state the famous theorem of Atiyah--Bott (see end of Chapter 9 in \cite{AtBott} for the unitary case, see Section 1.8 in Goldman's paper \cite{Goldman.2} for the general case):
\begin{thm}[Atiyah--Bott \cite{AtBott}]\label{Thm:AB}
The action of the gauge group on the space of connections is Hamiltonian and the moment map is given by the curvature. Thus, the Hamiltonian reduction $\mathcal{A}//\mathcal{G}$ is the moduli space of flat connections.
\end{thm}

Note that the moment map $m:\mathcal{A}\to\mathrm{Lie}(\mathcal{G})^*=\Omega^0(S,\g)^*$ can be identified with the curvature $F(A)\in\Omega^2(S,\g)$ using the non-degenerate pairing  
$\langle \alpha,\beta\rangle = \int_{S} \tr \alpha \beta$, for $\alpha\in \Omega^2(S,\g)$ and $\beta\in\Omega^0(S,\g)$.

%\newpage

\subsection{$L$-parabolic reduction}\label{settingparab}
Let $S$ be a closed surface. We consider a complex vector bundle $V$ of rank $n$ over $S$ equipped with a volume form and a line subbundle $L\subset V$. After some general considerations, we shall equip $S$ with a complex structure and consider $L$ to be a holomorphic line bundle isomorphic to $K^{(n-1)/2}$ (where $K$ is the canonical bundle).
 
We want to mimic the Atiyah--Bott reduction for $G=\SL_n(\mathbb{C})$ on $V$ with the extra constraint of preserving $L$. 
As before, let $\mathcal{A}$ denote the space of all $\mathfrak{sl}_n(\C)$-connections on $V$ and $\mathcal{G}$ be the space of $\mathrm{SL}_n(\C)$-gauge transformations. Consider the subgroup $\mathcal{P}_L\subset \mathcal{G}$ consisting of gauge transformations preserving $L$. 
We want to compute and analyze the Hamiltonian reduction $\mathcal{A}//\mathcal{P}_L$. 
\begin{definition}
A representative in $\mathcal{A}$ of an element of $\mathcal{A}//\mathcal{P}_L$ is called \textbf{$L$-parabolic connection}.
\end{definition}

Since $\mathcal{P}_L\subset \mathcal{G}$, we know by the Atiyah--Bott Theorem \ref{Thm:AB} that the action of $\mathcal{P}_L$ on the space of connections $\mathcal{A}$ is Hamiltonian with moment map $m_L: d_A\mapsto i^*F(A)$ where $i: \mathcal{P}_L \hookrightarrow \mathcal{G}$ is the inclusion and $i^*: \Lie(\mathcal{G})^* \twoheadrightarrow \Lie(\mathcal{P}_L)^*$ the induced surjection on the dual Lie algebras. Therefore, $d_A\in\mathcal{A}$ is $L$-parabolic if and only if $m_L(d_A)=0$.

\begin{prop}\label{Prop:L-para-charac}
A connection $d_A\in\mathcal{A}$ is $L$-parabolic if and only if for all infinitesimal gauge transformations $M\in\mathcal{C}^\infty(S,\mathfrak{sl}(V))$ preserving $L$, we have $\int_{\S} \tr F(A)M = 0$.
\end{prop}
\begin{proof}
This follows from the definitions: $d_A$ is $L$-parabolic if and only if $m_L(d_A)=0$. This is equivalent to $\int_{\S} \tr F(A)M=0$ for all $M\in \mathrm{Lie}(\mathcal{P}_L)$, which are infinitesimal gauge transformations preserving $L$.
\end{proof}

\begin{coro}\label{prop:rk-1}
The curvature of an $L$-parabolic connection is of rank at most 1.
\end{coro}
Indeed, the curvature has to be concentrated in the trace-dual of $L$, which is of rank 1. We will see this more concretely below in Section \ref{Sec:local-para-conn}, where we study the local structure of $L$-parabolic connections.

We want to find a preferred representative of an $L$-parabolic connection $d_A$. To do so, consider a complex vector field $v\in \Gamma(T^{\C}\S)$ and put $\nabla=(d_A)_v$. For a section $s\in\Gamma(L)$, consider $B=(s,\nabla s, \nabla^2s,...,\nabla^{n-1}s)$. On the locus where $\det(B)\neq 0$, we get a trivialization of $V$. Changing $s$ to $fs$, where $f\in\mathcal{C}^{\infty}(\S,\C^*)$, the determinant changes to 
$$\det(fs,\nabla(fs),...,\nabla^{n-1}(fs))=f^{n}\det(B).$$
Hence up to multiplication by an $n$-th root of unity, there is a unique  section $s\in \Gamma(L)$ such that $B$ is a basis of unit volume (on the locus where $\det(B)\neq 0$).
We can use this basis to get a representative of $d_A$. This depends of course on the vector field $v$. A typical example is to consider a complex structure on $\S$ and to take the vector field given in local complex coordinates by $v=\tfrac{\partial}{\partial z}$.

\subsection{Local coordinates}\label{Sec:local-para-conn}

We investigate the local structure of $L$-parabolic connections. For that, fix a complex structure on $S$, consider an open chart $U\subset S$ with complex coordinates $(z,\bar{z})$. Denote by $(\partial=\tfrac{\partial}{\partial z}, \bar{\partial}=\tfrac{\partial}{\partial \bar{z}})$ the induced basis on $T^\C U$. We can choose a basis of $V$ restricted to $U$ such that the first basis vector generates $L$. Hence the subgroup $\mathcal{P}_L(U)=\mathcal{P}_L\!\mid_U$ is the space of functions from $U$ to the space of matrices with determinant 1 of the form 
\begin{equation}\label{Eq:form-P}
\begin{pmatrix}
* & * & \cdots & * \\
0 & * & \cdots & * \\
\vdots & \vdots & & \vdots \\
0 & *& \cdots & * \\
\end{pmatrix}.
\end{equation}

The map $i^*:\Lie(\mathcal{G}(U))^* \twoheadrightarrow \Lie(\mathcal{P}_L(U))^*$ is explicitly given by forgetting the last $n-1$ entries in the first row. This means that $m_L^{-1}(\{0\})$ is the space of all $d_A \in \mathcal{A}(U)$ such that the curvature $F(A)$ is of the form 
\begin{equation}\label{Eq:parab-curvature}
\begin{pmatrix}
 0& \xi_2 & \cdots & \xi_n\\
 0 & 0 & \cdots &0\\
 \vdots & \vdots &  & \vdots\\
 0& 0 &\cdots & 0 \\
\end{pmatrix},
\end{equation}
with $\xi_k\in\mathcal{C}^\infty(U,\C)$ for $k=2,...,n$. This proves in particular Corollary \ref{prop:rk-1}.

To get a parametrization of $L$-parabolic connections on $U$, consider $d_A\in \mathcal{A}_L(U)$ and decompose it as $d+A=d+A_1dz + A_2d\bar{z}$, where $A_1, A_2\in \Omega^0(U,\mathfrak{sl}_n(\C))$. Set $\nabla = (d_A)_{\partial} = \partial + A_1$ and $\bar{\nabla} = (d_A)_{\bar\partial}= \bar{\partial}+ A_2$.

From the end of the previous subsection, we know that there is a unique (up to overall multiplication by an $n$-th root of unity) section $s\in\Gamma(L\!\mid_U)$ such that $B=(s,\nabla s,...,\nabla^{n-1}s)$ is a basis of unit volume (at the points where $\det B\neq 0$). Since the first vector of the initial basis generated $L$, the gauge transformation towards the new basis $B$ is in $\mc{P}_L$. 
Note that in basis $B$, the connection matrix $A_1$, which is the matrix of the operator $\nabla$, becomes a companion matrix:
\begin{equation}\label{firstmatrix}
\begin{pmatrix} 
0&  &  & \bm\hat{t}_n\\
1 &0   &  &\vdots \\
& \ddots &\ddots &\bm\hat{t}_2 \\
 &  & 1 & 0 \\
\end{pmatrix},
\end{equation}
where $\bm\hat{t}_k$ are complex-valued functions on $U$ which will serve for the parametrization of $L$-parabolic connections.

We only have to determine under which condition $\det(B)\neq 0$ on $U$. This condition depends on the connection matrix $A_1$.

\begin{prop}\label{existence-para-gauge}
Let $d+A_1dz+A_2d\bar{z}$ be a connection on $U$. Denote by $a_{i,j}$ the entries of $A_1$. Define the column vector $a=(a_{i,1})_{i=2,...,n}$ and the differential operator $\widetilde{\nabla}=\partial+(a_{i,j})_{2\leq i,j\leq n}$. Then $\det(B)\neq 0$ on $U$ if and only if $\det(a,\widetilde{\nabla}a,...,\widetilde{\nabla}^{n-2}a)\neq 0$.
\end{prop}

\begin{proof}
Since $A_1$ is the matrix of the operator $\nabla$, the condition $\det B\neq 0$ is equivalent to finding a companion matrix \eqref{firstmatrix} in the $\mathcal{P}_L$-gauge orbit of $A_1$.
In equations, we look for $P\in\mathcal{P}_L$ such that $P^{-1}A_1P+P^{-1}\partial P=M$ is a companion matrix. This is equivalent to $\nabla P=PM.$

To set some notations, we write $P$ as:
$$P= \begin{pmatrix}
	p_{11} & p_{12} & \cdots &p_{1n} \\
	0 & &  &  \\
	\vdots & C_2& \cdots& C_n \\
	0 &  &  & 
\end{pmatrix},$$
where $C_k$ denotes the $k$\textsuperscript{th} column in the matrix without the first entry (i.e. a column vector of length $n-1$).
We can then compute $PM$:
$$PM=\begin{pmatrix}p_{12}&\cdots &p_{1n}&*\\ &&&* \\ C_2&\cdots &C_n&\vdots \\ &&&*\end{pmatrix},$$
where the stars denote some entries we will not need. This gives $n^2-n$ equations by the first $n-1$ columns.

Our strategy is the following: first we express $p_{ij}$ for $j>1$ in terms of $a_{ij}$ (the ``constants'') and $p_{11}$ (and its derivatives). Second, we get an expression of $p_{11}$ in terms of $a_{ij}$. In this second step, the condition of the proposition appears.

The matrix equation $\nabla P=PM$ gives for all $i,j$ with $j<n$:
\begin{equation}\label{auxmat}
\sum_{k=1}^n a_{ik}p_{kj}+\partial p_{ij} = p_{i,j+1}.
\end{equation} 
This equation expresses $p_{i,j+1}$ in terms of constants and $(p_{k,j})_{1\leq k\leq n}$. Applying this recursively and since $p_{1,j}=0$ for $j>1$, we get the first step.

To achieve our second goal, we prove the following:
\begin{lemma}
Denote by $P_0$ the square-submatrix of $P$ consisting in the columns $C_2, ..., C_n$. We then have $\det(P_0) = p_{11}^{n-1}\det(a,\widetilde{\nabla}a,...,\widetilde{\nabla}^{n-2}a).$
\end{lemma}
\begin{proof}
The first column of the equation $\nabla P=PM$ gives $C_2=p_{11}a$. From the $k$-th column of $\nabla P=PM$ (with $2\leq k\leq n-1$) and the definition of $\widetilde{\nabla}$, we get $$C_{k+1}=\widetilde{\nabla}(C_k)+p_{1k}\, a.$$
Using operations on columns, we can transform $C_{k+1}$ to $p_{11}\widetilde{\nabla}^{k-1}(a)$ for $1\leq k\leq n-1$. The lemma follows.
\end{proof}

Since $1=\det P =p_{11}\det P_0$, we get $$p_{11} = \det(a,\widetilde{\nabla}a,...,\widetilde{\nabla}^{n-2}a)^{-\frac{1}{n}}.$$
Therefore, if the $L$-parabolic gauge exists, then the previous determinant is non-zero.

Conversely, if this determinant is non-zero, we can compute $P$ from $A_1$, giving a transition matrix to a companion form, in which the basis is of the form $B=(s,\nabla s, ..., \nabla^{n-1}s)$. Since $B$ is a basis, we have $\mathrm{det}(B)\neq 0$.
\end{proof}

We denote by $\mathcal{A}^*_L(U)$ the connections on $U$ satisfying the condition of Proposition \ref{existence-para-gauge}. Elements in $\mathcal{A}^*_L(U)$ will be called $\mathbf{L}$\textbf{-generic}. Note that this notion depends on the complex structure on $U$, but not on the chosen coordinates. Indeed, in another complex coordinate system $(w,\bar{w})$, the term $A_1dz$ becomes $A_1\tfrac{dz}{dw}dw$, so the column vector $a$ is multiplied by the non-vanishing function $\tfrac{dz}{dw}$, so the determinant is multiplied by $\left(\tfrac{dz}{dw}\right)^{n-1}\neq 0$.

By definition, to an $L$-generic connection we can associate the basis $B=(s,\nabla s,...,\nabla^{n-1}s)$. In this basis, the connection matrix $A_1$ becomes a companion matrix \eqref{firstmatrix}. For an $L$-generic connection $d_A$, 
we call $\mathbf{L}$\textbf{-parabolic gauge} a local trivialization of the bundle such that $A_1$ is a companion matrix.

To get coordinates on $\mathcal{A}^*_L//\mathcal{P}_L(U)$, we can use the $\mathcal{P}_L$-gauge freedom to put $A_1$ into companion form \eqref{firstmatrix}. Then, we have to analyze the second connection matrix $A_2$. Since $A_2$ is the operator $\bar{\nabla}$ in the basis $B$, the first column of $A_2$ is given by $\bar{\nabla}s$. We will show that knowing $\nabla^n s$ and $\bar{\nabla} s$ in the basis $B$ gives the desired coordinates. We define

\begin{equation}\label{eqqq2}
\bar{\nabla}s = \bm\hat{\mu}_1 s + \bm\hat{\mu}_2 \nabla s + ... + \bm\hat{\mu}_n \nabla^{n-1}s,
\end{equation}
where the $\bm\hat{\mu}_k$ are complex-valued functions on $U$. Equation \eqref{firstmatrix} implies

%A parabolic connection verifies $[\nabla, \bar{\nabla}]\nabla^i s = \xi_{i+1}$ for $i=0,1,...,n-2$ (where we put $\xi_1=0$), since $[\nabla, \bar{\nabla}] = F(A)$ is the curvature which is concentrated on the first row. It follows that $\bar{\nabla}\nabla^i s = \nabla^i \bar{\nabla} s$ for all $i=1,...,n-1$. Thus, the parabolic connection is fully described by $\nabla^n s$ and $\bar{\nabla}s$. We can write these expressions in the basis $B$:

\begin{equation}\label{eqqq1}
\nabla^n s = \bm\hat{t}_{n}s + \bm\hat{t}_{n-1} \nabla s+...+\bm\hat{t}_{2}\nabla^{n-2}s.
\end{equation}
Note that the functions $(\bm\hat{\mu}_k,\bm\hat{t}_k)_{2\leq k\leq n}$ do not depend on the choice of $B$ (which is unique up to multiplication by an $n$-th root of unity). Indeed, replacing $s$ by $\omega s$, where $\omega^n=1$, multiplies both sides of Equation \eqref{eqqq2} and \eqref{eqqq1} by $\omega$. 

%\begin{Remark} Notice the similarity between Equations \eqref{eqqq1} and \eqref{eqqq2}, and the ideal from Equation \eqref{ideal-I}. If we formally replace $\nabla$ by $p$ and $s$ by $1$, we get precisely the generators of the ideal from Equation \eqref{ideal-I}.
%\end{Remark}

\begin{prop}\label{Prop:para-L-para-conns}
The functions $(\bm\hat{\mu}_2, ..., \bm\hat{\mu}_n, \bm\hat{t}_{2}, ..., \bm\hat{t}_{n})$ parameterize $\mathcal{A}^*_L//\mathcal{P}_L(U)$, the space of $L$-generic parabolic connections on $U$.
\end{prop}

\begin{proof}
It is clear that these functions are independent. Thus, we have to show that $\bm\hat{\mu}_1$ and all the entries in $A_2$ which are not in the first column are determined from $(\bm\hat{\mu}_k,\bm\hat{t}_k)_{2\leq k\leq n}$.

Denote by $C_k$ the $k$-th column of $A_2$ ($1\leq k\leq n$). The curvature $F(A)=[\nabla,\bar\nabla]$ is constrained by Equation \eqref{Eq:parab-curvature}, which states that $[\nabla,\bar\nabla]\nabla^{k-1}s=\xi_k s$ for $1\leq k\leq n$ (where we put $\xi_1=0$). Hence for $2\leq k\leq n-1$, we get
\begin{equation}\label{Eq:xi-k}
\xi_k s = [\nabla,\bar\nabla]\nabla^{k-1}s = \nabla C_k-C_{k+1},
\end{equation}
since $A_2$ is the matrix of $\bar\nabla$ in basis $B$. We can use this equation to determine $C_{k+1}$ in terms of $C_k$ and the curvature $\xi_k$. Note that for $k\geq 2$, only the first $k-2$ entries of $C_k$ contain $\xi_\ell$'s (and their derivatives). In addition, the $m$-th entry of $C_k$ with $1\leq m\leq k-2$ contains a term $\xi_{k-m}$ and all other $\xi_\ell$ appearing there satisfy $\ell<k-m$.

The variable $\bm\hat{\mu}_1$ can be determined by the requirement that $A_2$ has to be traceless. Since there are no $\xi$-terms on the main diagonal of $A_2$, we get an equation for $\bm\hat{\mu}_1$ in terms of
$(\bm\hat{\mu}_2, ..., \bm\hat{\mu}_n, \bm\hat{t}_{2}, ..., \bm\hat{t}_{n})$ and their derivatives.

Finally, the curvature constraint gives one more equation:
\begin{equation}\label{Eq:xi-n}
\xi_n s = \nabla C_n-\bar\nabla(\nabla^n s)  = \nabla C_n-\bar\nabla(\bm\hat{t}_2\nabla^{n-2}s+...+\bm\hat{t}_n s)= \nabla C_n-\sum_{k=2}^n (\bar\partial\bm\hat{t}_k\nabla^{n-k}s+\bm\hat{t}_k C_{n+1-k}).
\end{equation}
By the analysis above, one sees that the coefficient in front of $\nabla^{n-k}s$ contains one term $\xi_k$ and all other terms are either our coordinates, or terms containing $\xi_\ell$ (or derivatives) with $\ell<k$. Hence, we can recursively express all $\xi_k$ in terms of $(\bm\hat{\mu}_2, ..., \bm\hat{\mu}_n, \bm\hat{t}_{2}, ..., \bm\hat{t}_{n})$. This in turn also determines $A_2$.
\end{proof}

\begin{example}\label{Ex:n2}
Consider the case $n=2$, i.e. we consider a parabolic $\SL(2, \C)$-connection $d_A=d+A_1dz+A_2d\bar{z}$, with $$A_1=\begin{pmatrix} 0 & \bm\hat{t}_2 \\ 1 & 0 \end{pmatrix} \;\;\text{ and }\;\; A_2=\begin{pmatrix} \bm\hat{\mu}_1 & \alpha_{1,2} \\ \bm\hat{\mu}_2 & \alpha_{2,2} \end{pmatrix}.$$

Let us compute the transformed matrix $A_2$. The second column $C_2$ is given by $\nabla C_1$, hence $\alpha_{1,2}=\partial\bm\hat{\mu}_1+\bm\hat{\mu}_2\bm\hat{t}_2$ and $\alpha_{2,2}=\partial\bm\hat{\mu}_2+\bm\hat{\mu}_1$. The trace of $A_2$ being zero gives $\bm\hat{\mu}_1=-\tfrac{1}{2}\partial\bm\hat{\mu}_2$. Thus:
$$A_2 = \left( \begin{array}{cc}
	-\frac{1}{2}\partial \bm\hat{\mu}_2 & -\frac{1}{2}\partial^2 \bm\hat{\mu}_2+\bm\hat{t}_2\bm\hat{\mu}_2 \\
	\bm\hat{\mu}_2 & \frac{1}{2}\partial \bm\hat{\mu}_2 
\end{array} \right).$$
The curvature is of the form $\left( \begin{smallmatrix} 0 & \xi_2 \\ 0 & 0 \end{smallmatrix} \right)$
where
$$\xi_2 = (\bar{\partial}-\bm\hat{\mu}_2\partial-2\partial\bm\hat{\mu}_2)\bm\hat{t}_2+\tfrac{1}{2}\partial^3\bm\hat{\mu}_2.$$
\end{example}

\begin{example}
Let us also compute the case $n=3$. The equations $C_2=\nabla C_1$ and $C_3=\nabla C_2+\xi_2 s$ determine the local form of the parabolic $\SL_3(\C)$-connection:
$$ d + \begin{pmatrix}0&0&\bm\hat{t}_3\\1&0&\bm\hat{t}_2\\0&1&0\end{pmatrix}dz
+\begin{pmatrix}\bm\hat{\mu}_1 & \partial\bm\hat{\mu}_1+\bm\hat{\mu}_3\bm\hat{t}_3 & \xi_2+\partial^2\bm\hat{\mu}_1+\partial(\bm\hat{\mu}_3\bm\hat{t}_3)+(\partial\bm\hat{\mu}_3+\bm\hat{\mu}_2)\bm\hat{t}_3\\
\bm\hat{\mu}_2 & \partial \bm\hat{\mu}_2+\bm\hat{\mu}_1+\bm\hat{\mu}_3\bm\hat{t}_2 & \partial^2\bm\hat{\mu}_2+2\partial\bm\hat{\mu}_1+\partial(\bm\hat{\mu}_3\bm\hat{t}_2)+\bm\hat{\mu}_3\bm\hat{t}_3+(\partial\bm\hat{\mu}_3+\bm\hat{\mu}_2)\bm\hat{t}_2\\
\bm\hat{\mu}_3 & \partial\bm\hat{\mu}_3+\bm\hat{\mu}_2&\partial^2\bm\hat{\mu}_3+2\partial\bm\hat{\mu}_2+\bm\hat{\mu}_1+\bm\hat{\mu}_3\bm\hat{t}_2\end{pmatrix}d\bar{z}.$$
The trace condition gives $\bm\hat{\mu}_1=-\tfrac{1}{3}(2\bm\hat{\mu}_3\bm\hat{t}_2+3\partial\bm\hat{\mu}_2+\partial^2\bm\hat{\mu}_3)$. Finally, the last equation $\xi_3s=\nabla C_3-\bar\nabla(\nabla^3s)$ allows to compute $\xi_2$ and $\xi_3$ in terms of the coordinates $(\bm\hat{\mu}_2,\bm\hat{\mu}_3,\bm\hat{t}_2,\bm\hat{t}_3)$.
\end{example}

\subsection{Global aspects}

Consider a finite open cover $\Sigma=\bigcup_{i=1}^N U_i$ of the compact Riemann surface $\Sigma$. Denote by $\pi_i:\mathcal{A}\to\mathcal{A}(U_i)$ the restriction maps. We consider
$$\mathcal{A}^*_L:= \bigcap_{i=1}^N \pi_i^{-1}(\mathcal{A}^*_L(U_i)),$$
i.e. the set of connections which are $L$-generic in all charts $U_i$. 
Elements in $\mathcal{A}^*_L$ are called $\mathbf{L}$\textbf{-generic} connections. 

The space of $L$-generic connections $\mathcal{A}^*_L$ is an open dense subset of $\mathcal{A}$ since it is a finite intersection of open dense subsets. It does not depend on the choice of open cover since $\pi_U(\mathcal{A}^*_L)\subset \mathcal{A}^*_L(U)$ for all open sets $U$ of $\Sigma$. However, it depends on the complex structure of $\Sigma$ and on the line subbundle $L$.

For an $L$-generic connection $d_A\in\mathcal{A}^*_L$, we know that on each $U_i$, there is a unique (up to $n$-th root of unity) section $s_i$ of $L\!\mid_{U_i}$ giving a unit-volume basis $B_i=(s_i,\nabla_i s_i, ..., \nabla_i^{n-1}s_i)$, where $\nabla_i=(d_A)_{\partial/\partial z_i}$ and $z_i$ is a complex coordinate on $U_i$.

Let us analyze the transition between two open charts, for instance between $U_1$ and $U_2$. For simplicity, put $z=z_1$ and $w=z_2$. Then on $U_1\cap U_2$, we have 
$$\nabla_2=(d_A)_{\partial/\partial w} = (d_A)_{\tfrac{dz}{dw}\partial/\partial z} = \tfrac{dz}{dw}\nabla_1.$$
On $U_1\cap U_2$ we can compare the two bases $B_1$ and $B_2$. The first vectors $s_1$ and $s_2$ are sections of the same line bundle, so there is a local function $f$ such that $s_2=fs_1$. Therefore
\begin{equation}\label{Eq:nature-L}
(s_2,\nabla_2 s_2,\nabla_2^2s_2,...)=(fs_1, \tfrac{dz}{dw}\nabla_1(fs_1), \tfrac{dz}{dw}\nabla_1\left(\tfrac{dz}{dw}\nabla_1(fs_1)\right),...).
\end{equation}
Hence the transition matrix between $B_1$ and $B_2$ is upper triangular with diagonal entries $\left(\tfrac{dz}{dw}\right)^kf$ for $k=0,...,n-1$. Since both $B_1$ and $B_2$ are of unit volume, the transition matrix has to be of determinant 1. Therefore we get $f=\left(\tfrac{dz}{dw}\right)^{(n-1)/2}$, which then gives
$$s_1(z)=\left(\frac{dw}{dz}\right)^{(n-1)/2}s_2(w).$$
This equation defines the transformation rule of a single section if the line bundle $L$ is the holomorphic bundle $K^{(n-1)/2}$, where $K$ denotes the canonical bundle of $\Sigma$. Note that if $n$ is even, we have to choose a spin structure on $\Sigma$, i.e. a choice of a square root of $K$.

From now on, we suppose $L=K^{(n-1)/2}$. Then, $s_1$ and $s_2$ come from the same section of $L$. The same holds for all other intersections. The only issue is that when we come back to $U_1$, there can be some monodromy which lies in the group of $n$-th roots of unity, since the section $s_1$ is unique up to multiplication by an $n$-th root of unity. To solve this problem, we consider all $n$ choices at once. Therefore, the $s_i$ define a \emph{global multi-valued section $s$ of $L=K^{(n-1)/2}$}.

\begin{Remark}
The section $s_i$ is non-vanishing on $U_i$ for all $i$, hence gluing them together gives a non-vanishing section. Since any section of $K^{(n-1)/2}$ has to vanish at some points, we see that when gluing the $s_i$ together, there is necessarily some monodromy which lies in the $n$-th roots of unity. This is why it is important to consider all $n$ sections at once.
\end{Remark}

Notice finally that the $\bm\hat{\mu}_k$ and $\bm\hat{t}_k$ are not tensors on $\S$ (their transformation rules under a coordinate change $z\mapsto w(z)$ are quite complicated). 
We will see in the following Section \ref{parabolicwithlambda} that if we introduce a parameter $\l$, we get tensors at the semiclassical limit.

\subsection{Change of line bundle}\label{symponconnections}

We describe an action on the space of connections by changing the line subbundle $L$ to another line subbundle $L'$. This action transforms $L$-parabolic connections to $L'$-parabolic connections. In particular it gives an action on the space of flat connections. 

\medskip
Consider two line subbundles $L$ and $L'$ of $V$, both holomorphic and isomorphic to $K^{(n-1)/2}$. Let $d_A\in \mathcal{A}^*_L\cap \mathcal{A}^*_{L'}$, a connection which is generic for $L$ and $L'$. Note that since $\Sigma$ is compact, Proposition \ref{existence-para-gauge} gives $\mathcal{A}^*_L\cap\mathcal{A}^*_{L'}\neq \emptyset$ for all $L'$ sufficiently close to $L$.

Locally on an open chart $U\subset \Sigma$, we have two sections $s\in\Gamma(L\!\mid_U)$ and $s'\in\Gamma(L'\!\mid_U)$ giving unit-volume bases $B=(s,\nabla s,...,\nabla^{n-1}s)$ and $B'=(s',\nabla s',...,\nabla^{n-1}s')$. Denote by $X$ the transition matrix between $B$ and $B'$. 
We use $X$ as gauge transformation on $d_A=d+A$. Thus, changing $L$ to $L'$ induces locally the action $X.A = XAX^{-1}+Xd\left(X^{-1}\right)$.

Notice that since $s$ and $s'$ are uniquely given up to multiplication of an $n$-th root of unity, so is $X$. The action on the connection is well-defined though, since $X$ and $\omega X$ with constant $\omega$ induce the same action on $A$.

In order to define a global action on $d_A$, we have to analyze the local description at the intersection of two charts $U_1$ and $U_2$. On $U_1\cap U_2$, there are four bases, which can be represented as follows (together with their transitions):

\vspace{0.2cm}
\begin{center}
\begin{tikzpicture}
\node[draw] (A) at (-4,1) {$B_1=(s_1(z),\nabla_z s_1(z),...,\nabla_z^{n-1}s_1(z))$};
\node[draw] (B) at (-4,-1) {$B_1'=(s_1'(z),\nabla_z s_1'(z),...,\nabla_z^{n-1}s_1'(z))$};
\node[draw] (C) at (4,1) {$B_2=(s_2(w),\nabla_w s_2(w),...,\nabla_w^{n-1}s_2(w))$};
\node[draw] (D) at (4,-1) {$B_2'=(s_2'(w),\nabla_w s_2'(w),...,\nabla_w^{n-1}s_2'(w))$};
\draw[->, >=latex] (A)--(B);
\draw[->, >=latex] (A)--(C);
\draw[->, >=latex] (C)--(D);
\draw[->, >=latex] (B)--(D);
\draw[above] (0,1) node {$P$};
\draw[above] (0,-1) node {$P'$};
\draw[right] (4,0) node {$X_2$};
\draw[right] (-4,0) node {$X_1$};
\end{tikzpicture}
\end{center}
\vspace{0.2cm}

\begin{prop}
With the notations from above, we have $P'=P$. 
\end{prop}

\begin{proof}
To compute $P$, we have to express the second basis $B_2$ in terms of the first $B_1$. Since $s_1$ and $s_2$ are two local expressions of a section of $L=K^{(n-1)/2}$, we know from Equation \eqref{Eq:nature-L} that $s_2(w)=fs_1(z)$ with $f(z)=\left(\tfrac{dz}{dw}\right)^{(n-1)/2}$. 
Since the coordinate change $z\mapsto w$ is holomorphic, we get $\nabla_w=\tfrac{dz}{dw}\nabla_z$. Hence $$\nabla_ws_2(w)=\tfrac{dz}{dw}\nabla_z(f(z)s_1(z))=\tfrac{dz}{dw}(f'(z)s_1(z)+f(z)\nabla_zs_1(z)).$$
A similar computation shows that the expression of $\nabla_w^{k}s_2(w)$ in $B_1$ only depends on $f$ and its derivatives.

Since $L'$ is also isomorphic to $K^{(n-1)/2}$, we get analogously $s_2'(w)=f(z)s_1'(z)$. Therefore, the transition matrix $P'$ is constructed the same way as $P$ with the same function $f$, i.e. $P'=P$.
\end{proof}

The proposition implies that $X_2=PX_1P^{-1}$, so the local gauge transformations $X_i$ glue together to form a well-defined global gauge transformation $X$. Thus, we get a global action of $X$ on $d_A$. Note that this action does not come from a group, but from the groupoid of pairs $(L,L')$ of line subbundles of $V$ isomorphic to $K^{(n-1)/2}$, where the composition is given by $(L,L')\circ (L',L'')=(L,L'')$.

\begin{prop}
If $d_A\in \mathcal{A}^*_L\cap\mathcal{A}^*_{L'}$ is $L$-parabolic, then $X.d_A$ is $L'$-parabolic.
\end{prop}
\begin{proof}
We use Proposition \ref{Prop:L-para-charac} to show that $X.d_A$ is $L'$-parabolic. Let $M'\in\mathcal{C}^\infty(S,\mathfrak{sl}(V))$ be an infinitesimal gauge transformation preserving $L'$. In particular, $M's'=\lambda s'$ with $\lambda\in\mathcal{C}^\infty(S,\C)$. Since $s'=Xs$, we get $M'Xs=\lambda Xs$. Thus, $X^{-1}M'X$ is an infinitesimal gauge transformation preserving $L$. We then conclude, using twice Proposition \ref{Prop:L-para-charac} and the assumption that $d_A$ is $L$-parabolic:
$$\int_\Sigma \tr(X.F(A) M') = \int_\Sigma \tr(XF(A)X^{-1}M')= \int_\Sigma \tr(F(A)X^{-1}M'X) =0.$$
\end{proof}

%Consider the groupoid consisting in pairs $(L,L')$, such that $\mathcal{A}_L^*\cap \mathcal_{L'}^*\neq \emptyset$, with operation $(L,L')\circ (L',L'') = (L,L'')$.
As a consequence, we see that we have a well-defined infinitesimal action\footnote{The action is by some Lie algebroid consisting in pairs $(L,\delta L)$, where $\delta L$ is an infinitesimal variation of $L$.} on the space of generic \emph{flat} connections, since they are $L$-parabolic for all line subbundles $L$. Let us describe with more detail the infinitesimal variation on the space of flat connections.

\medskip
Let $d_A$ be an $L$-generic flat connection. For all $L'$ sufficiently close to $L$, $d_A$ is also $L'$-generic, hence we have an action by gauge transformations on $d_A$. In particular, we have a well-defined infinitesimal action.

An infinitesimal change of $L$ can be described by the change $\delta s$ of the canonical section $s$, which we can express in the basis $B$:
\begin{equation}\label{Eq:change-L}
\delta s = v_1 s+v_2\nabla s +...+v_{n}\nabla^{n-1}s = \bm\hat{H}s,
\end{equation}
where $\bm\hat{H}=v_1+v_2\nabla+...+v_{n}\nabla^{n-1}$ is a differential operator of degree $n-1$. Multiplying $s$ by an $n$-th root of unity does not change the operator $\bm\hat{H}$, hence Equation \eqref{Eq:change-L} has a global meaning. It generates an infinitesimal action on the whole basis $B$, and thus gives an infinitesimal automorphism of the bundle $V$.

Let us describe how to compute the matrix $X$ describing the infinitesimal base change induced by $\delta s$. 
Write the base change as $$(s,\nabla s,...,\nabla^{n-1}s) \mapsto (s,\nabla s,...,\nabla^{n-1}s)+\varepsilon (\delta s,\nabla \delta s,...,\nabla^{n-1}\delta s).$$ So the first column of $X$ is just given by $Xs = \delta s = v_1s+v_2\nabla s+...+v_n\nabla^{n-1}s$. The second is given by $X\nabla s = \nabla \delta s = \nabla(v_1s+v_2\nabla s+...+v_n\nabla^{n-1}s)$. 
We notice that the construction of this matrix $X$ is exactly the same as for the matrix $A_2$ with the only difference that the variables in $A_2$ are called $\bm\hat{\mu}_k$ instead of $v_k$. Since both matrices are traceless, the terms $v_1$ and $\bm\hat\mu_1$ also correspond to each other. Therefore we have proven: 
\begin{prop}\label{computeX}
In a local chart, the matrix $X$ of the gauge transformation coming from an infinitesimal variation \eqref{Eq:change-L}, is given by $$X=A_2 \mid_{\bm\hat{\mu}_k\mapsto v_k}.$$
\end{prop}

Let us compute the action induced by $X$ on the parameters $(\bm\hat{t}_k, \bm\hat{\mu}_k)_{2\leq k\leq n}$. Recall that the parameters $\bm\hat{t}_k$ are given by the relation $\nabla^n s = \bm\hat{P}(\nabla)s$ where $\bm\hat{P}(\nabla)=\bm\hat{t}_2\nabla^{n-2}+...+\bm\hat{t}_n$. Similarly, the parameters $\bm\hat{\mu}_k$ are given by $\bar{\nabla}s = \bm\hat{Q}(\nabla)s$ where $\bm\hat{Q}(\nabla) = \bm\hat{\mu}_1+\bm\hat{\mu}_2\nabla...+\bm\hat{\mu}_n\nabla^{n-1}$.

\begin{prop}\label{variation-infinit}
The infinitesimal action induces the variations $\delta \bm\hat{P}$ and $\delta \bm\hat{Q}$ given by
\begin{align*}
\delta \bm\hat{P} &= [\bm\hat{H}, -\nabla^n+\bm\hat{P}] \mod \bm\hat{I} \\
\delta \bm\hat{Q} &= [\bm\hat{H}, -\bar{\nabla}+\bm\hat{Q}] \mod \bm\hat{I}
\end{align*}
where $\bm\hat{I} = \langle \nabla^n-\bm\hat{P}, -\bar{\nabla}+\bm\hat{Q} \rangle$ is a left-ideal of differential operators.
\end{prop}

\begin{proof}
The proposition follows directly from the observation that the variation $\delta \bm\hat{P}$ by definition satisfies
$$\nabla^n (s+\varepsilon \bm\hat{H}s) = (\bm\hat{P}+\varepsilon \delta \bm\hat{P})(s+\varepsilon \bm\hat{H}s).$$
Expanding gives $\nabla^n(\bm\hat{H}s) = \delta \bm\hat{P}(s)+\bm\hat{P}(\bm\hat{H}s)$ since $\nabla^n(s) = \bm\hat{P}(s)$. Computing differential operators modulo $\bm\hat{I}$ is equivalent to applying them to $s$. We then get
$$[\bm\hat{H}, -\nabla^n+\bm\hat{P}](s) = \bm\hat{H}(-\nabla^n(s)+\bm\hat{P}(s))+\nabla^n(\bm\hat{H}s)-\bm\hat{P}(\bm\hat{H}s) = \delta \bm\hat{P}(s).$$
The same applies to the variation of $\delta \bm\hat{Q}$.
\end{proof}

\begin{Remark}
Notice the similarity between the formulas of the previous proposition and Equations \eqref{idealvariation} describing the variation of a higher complex structure. As we will see below in Theorem \ref{actionsymponlambdaconn}, one could say that the variation of higher complex structures is the semiclassical limit of the variation on connections induced from changing $L$.
\end{Remark}

\begin{Remark}
Proposition \ref{computeX} implies that an infinitesimal gauge transformation can be decomposed uniquely into an infinitesimal gauge transformation preserving $L$ and a transformation coming from an infinitesimal change of $L$. Thus the moduli space of flat connections $\mathcal{A}^{\text{flat}}/\mathcal{G}$ behaves like a double quotient, where one first quotients by $\mathcal{P}_L$ and then by the groupoid action changing $L$.
\end{Remark}

\section{$L$-parabolic connections and higher complex structures}\label{parabolicwithlambda}

In this section, we analyze a special family of $L$-parabolic connections linked to higher complex structures. In the semiclassical limit, we show that from the coordinates of $L$-parabolic connections, we recover a higher complex structure with a $\mu$-holomorphic covector.

%Our main theorems give gauge-theoretic interpretations of the $\mu$-holomorphicity term and higher diffeomorphisms. First we prove that flatness of the family of connections implies the covector to be $\mu$-holomorphic. Second, we show that the infinitesimal action by higher diffeomorphisms on higher complex structure with covector is equivalent to the action on flat connections induced by changing the line bundle.

\subsection{Setting and parametrization}\label{param-lambda}

Consider a closed surface $\S$ equipped with a higher complex structure $I$. This induces a complex structure on $\S$, making it into a Riemann surface $\Sigma$ with local coordinates $(z,\bar{z})$. From Section \ref{indbundle}, we know that $I$ induces a bundle $V$ together a line subbundle $L\subset V$, which is holomorphic and isomorphic to $K^{(n-1)/2}$. We also choose a volume form on $V$. Thus we are in the setting of the previous section.

We consider a family of $L$-parabolic connections on $V$, indexed by $\l\in\C^*$, of the form
\begin{equation}\label{Eq:3-term-conn}
C(\l)=\l \Phi + d_A + \l^{-1}\Psi,
\end{equation}
where $\Phi$ is the induced field on $V$ from the higher complex structure, $d_A$ is a connection on $V$ satisfying $d_A(\Phi)=0$, and $\Psi\in\Omega^1(S,\mathfrak{sl}(V))$.  

An example of such a family is when $C(\l)$ is a family of \emph{flat} connections. Then $d_A(\Phi)=0$ is automatic (see Equation \eqref{Eq:curv-c-lambda} below), and flat connections are of course $L$-parabolic. 
%Later in Section \ref{finalstep}, we fix a Hermitian structure $h$ on $V$ and impose $\Psi=\Phi^{*_h}$ and $d_A$ unitary. Then, for generic $h$, $C(\l)$ is $L$-parabolic if and only if it is flat (see Proposition \ref{Prop:flat-para-equiv}).

The curvature of $C(\l)$ is given by
\begin{equation}\label{Eq:curv-c-lambda}
F(C(\l))=dC(\l)+[C(\l)\wedge C(\l)] = \l^2\Phi\wedge\Phi+\l d_A(\Phi)+F(A)+[\Phi\wedge\Psi]+\mathcal{O}(\l^{-1}),
\end{equation}
where $F(A)$ denotes the curvature of $d_A$. We have $\Phi\wedge\Phi=0$, since $\Phi$ comes from a higher complex structure, so it can be written locally as $\Phi=\Phi_1dz+\Phi_2d\bar{z}$ with $[\Phi_1,\Phi_2]=0$. Since we assume $d_A(\Phi)=0$, the curvature starts with a constant term:
\begin{equation}\label{Eq:curv-3-term}
F(C(\l))=F(A)+[\Phi\wedge\Psi]+\mathcal{O}(\l^{-1}).
\end{equation}

On a local chart $U\subset \S$, we define 
$$C_1(\l)=\l \Phi_1 + A_1 +\l^{-1}\Psi_1 \;\;\text{ and }\;\; C_2(\l)=\l \Phi_2 + A_2 + \l^{-1}\Psi_2,$$
as the $(1,0)$-part and $(0,1)$-part of the connection matrix of $C(\l)$. We use coordinates in the fibers of $V\!\mid_U$ such that $\Phi$ is in the standard form given in Equation \eqref{Eq:matrix-Phi}. We also define $\nabla=\partial + C_1(\l)$ and $\bar{\nabla}=\bar{\partial}+C_2(\l)$, where $(\partial=\tfrac{\partial}{\partial z},\bar\partial=\tfrac{\partial}{\partial\bar{z}})$ is a local basis of $T^\C U$.

\begin{prop}
For sufficiently large $\l$, the connection $C(\l)$ is $L$-generic.
\end{prop}
\begin{proof}
Consider an open chart $U\subset \Sigma$. We prove that $C(\l)$ is $L$-generic on $U$ using the criterion from Proposition \ref{existence-para-gauge}. Denote by $c_{ij}$ the matrix entries of $C_1(\l)$. Then we consider the column vector $a=(c_{i1})_{2\leq i\leq n}$. By the form of $\Phi_1$, we know that $c_{21}=\lambda+\mathcal{O}(1)$ and $c_{i1}=\mathcal{O}(1)$ for all $i=3,...,n$.

Define $\widetilde{\nabla}=\partial+M$ with $M=(c_{ij})_{2\leq i,j\leq n}$. Again from the explicit form of $\Phi_1$, we see that all entries of $\widetilde{\nabla}(a)$ are of order $\mathcal{O}(\l)$ apart from the second entry which is $\lambda^2+\mathcal{O}(\l)$. Similarly for $1\leq k\leq n-2$, all entries of $\widetilde{\nabla}^k(a)$ are of order $\mathcal{O}(\l^k)$, apart from the entry number $k+1$ which is $\l^{k+1}+\mathcal{O}(\l^k)$. Hence for large $\l$, $\det(a,\widetilde{\nabla}(a),...,\widetilde{\nabla}^{n-2}(a))\neq 0$, since the dominant term is $\l^{(n-1)n/2}$.
\end{proof}

The proposition implies that for large $\l$ (or treating $\l$ as formal variable), there is a preferred representative of $C(\l)$ in its $\mathcal{P}_L$-gauge orbit. This representative is locally of the form
\begin{equation}\label{paragauge}
d+\left(\begin{array}{cccc}
	 0& & & \bm\hat{t}_n(\l) \\
	1& 0& & \vdots \\
	 & \ddots &\ddots & \bm\hat{t}_2(\l) \\
	 & & 1& 0 
\end{array}\right)
dz + \left(\begin{array}{cc}
	 \bm\hat{\mu}_1(\l) &  \\
	 \bm\hat{\mu}_2(\l)&    \bm\hat{\alpha}_{ij}(\l)\\
	 \vdots &  \\
	 \bm\hat{\mu}_n(\l) & 
\end{array}\right) d\bar{z}
\end{equation}
where $\bm\hat{\alpha}_{ij}(\l)$ and $\bm\hat{\mu}_1(\l)$ are explicit functions of the other variables (since $C(\l)$ is $L$-parabolic, see Proposition \ref{Prop:para-L-para-conns}). 

Recall from the previous section that this form comes from a local unit-volume basis of the form $B=(s,\nabla s,...,\nabla^{n-1}s)$, where $s$ is a section of $L$, and depending on $\l$. The functions $(\bm\hat{\mu}_k,\bm\hat{t}_k)_{2\leq k\leq n}$, depending on $z, \bar{z}$ and on $\l$, are characterized by
\begin{equation}\label{paraboliccoord1}
\nabla^n s = \bm\hat{t}_{n}(\l)s + \bm\hat{t}_{n-1}(\l) \nabla s+...+\bm\hat{t}_{2}(\l)\nabla^{n-2}s ,
\end{equation}
\begin{equation}\label{paraboliccoord2}
\bar{\nabla}s = \bm\hat{\mu}_1(\l) s + \bm\hat{\mu}_2(\l) \nabla s + ... + \bm\hat{\mu}_n(\l) \nabla^{n-1}s,
\end{equation}
and give a local parametrization by Proposition \ref{Prop:para-L-para-conns}.

%We can compute the $\bm\hat{\alpha}_{ij}(\l)$ using $\bar{\nabla}\nabla^ks = \nabla^k\bar{\nabla}s$ for $k\leq n-1$ which holds since the curvature $[\nabla, \bar{\nabla}]$ is zero. The only relations between the parameters $\bm\hat{\mu}_k(\l)$ and $\bm\hat{t}_k(\l)$ come from the flatness constraint: the parabolic curvature is zero, i.e. $\xi_k(\lambda)=0$ for all $k=2,3,...,n$ (see Equation \eqref{Eq:parab-curvature} for the definition of the $\xi_k$).

\begin{example}\label{examplen2}
We analyze the case $n=2$ on an open chart $U$. Consider a connection $C(\l)=\l\Phi+d+A+\l^{-1}\Phi^\dagger$, where we fixed a Hermitian structure on the bundle and $\Phi^\dagger$ denotes the conjugated transpose. We consider the case where $\Phi_1 = \left(\begin{smallmatrix} 0 & 0 \\ b_1 & 0\end{smallmatrix}\right)$, $A_1 = \left(\begin{smallmatrix} a_0 & a_1 \\ a_2 & -a_0\end{smallmatrix}\right)$, $\Phi_2=\mu_2\Phi_1$. Further we impose $A_2 = -A_1^{\dagger}$. So we have 
$$C_1(\l)=\begin{pmatrix} a_0 & a_1+\l^{-1}\bar{\mu}_2\bar{b}_1 \\ a_2+\l b_1 & -a_0\end{pmatrix}\; \text{,} \; C_2(\l)=\begin{pmatrix} -\bar{a}_0 & -\bar{a}_2+\l^{-1}\bar{b}_1 \\ -\bar{a}_1+\l\mu_2b_1 & \bar{a}_0\end{pmatrix}.$$
The condition $d_A(\Phi)=0$ gives $a_2=-\bar{\mu}_2\bar{a}_1$ and $(-\bar\partial+\mu_2\partial+\partial\mu_2-2(\bar{a}_0+\mu_2a_0))b_1=0$.
We look for $P=\left(\begin{smallmatrix} p_1 & p_2 \\ 0 & 1/p_1\end{smallmatrix}\right)$ such that 
$$P C_1(\l)P^{-1}+ P\partial P^{-1} = \begin{pmatrix} 0 & \bm\hat{t}_2(\l) \\ 1 & 0\end{pmatrix}.$$
Multiplying by $P$ from the right, one can solve the system. One finds $p_1=(\l b_1+a_2)^{1/2}$ and $p_2=-\frac{a_0}{p_1}+\frac{\partial p_1}{p_1^2}$. Hence 
$$\bm\hat{t}_2(\l)=\l a_1b_1+ \text{ constant term }-\l^{-1}\bar{\mu}_2^2 \bar{a}_1\bar{b}_1.$$ Transforming $C_2(\l)$ with $P$ we get 
$$\bm\hat{\mu}_2(\l) = \frac{-\bar{a}_1+\l\mu_2 b_1}{\l b_1-\bar{\mu}_2\bar{a}_1}.$$
We can study the behavior of $\bm\hat{\mu}_2(\l)$ for $\l$ near $\infty$ and 0. For $\l\rightarrow \infty$, the Taylor expansion of $\bm\hat{\mu}_2(\l)$ yields
$$\bm\hat{\mu}_2(\l)=\mu_2+(\mu_2\bar{\mu}_2-1)\sum_{k=1}^\infty \frac{\bar{\mu}_2^{k-1}\bar{a}_1^k}{b_1^k}\l^{-k}.$$
For $\l \rightarrow 0$, we get a similar expression, with leading term $1/\bar{\mu}_2={}_2\mu$, the conjugated complex structure from Section \ref{dualcomplexstructure}.
\end{example}

The example shows several phenomena which are true in general: 
\begin{prop}\label{parametrisationlambda}
The $\bm\hat{\mu}_k(\l)$ are rational functions in $\l$. The highest term in $\l$ when $\l \rightarrow \infty$ is $\l^{2-k}\mu_k$ where $\mu_k$ is the higher Beltrami differential from the higher complex structure $I$. 
The $\bm\hat{t}_k(\l)$ are also rational functions in $\l$. For $\l \rightarrow \infty$, the highest term is given by $\l^{k-1}t_k$ where 
\begin{equation}\label{Eq:t_k-form}
t_k=\tr(\Phi_1^{k-1}(\partial+A_1)).
\end{equation}
\end{prop}

We will see later that $(\mu_k, t_k)_{2\leq k\leq n}$ is a representative of a point in the cotangent bundle $\cotang$, which justifies the notation. The formula \eqref{Eq:t_k-form} for $t_k$ is the trace of the composition of $\Phi_1^{k-1}$ with the operator $\partial+A_1=(d_A)_{\partial}$. Note that this equation is gauge invariant, but depends on the local coordinate $z$.

\begin{Remark}
If we equip $V$ with a Hermitian structure $h$ and consider $C(\l)=\l\Phi+d_A+\l^{-1}\Phi^{*_h}$ with $d_A$ Hermitian, then we get the same properties for $\l\to 0$, where the conjugated higher complex structure with a cotangent vector from Section \ref{dualcomplexstructure} appears. 
\end{Remark}

To prove the proposition, we need the following lemma. Recall that $s$ depends on $\l$.
\begin{lemma}\label{Lemma:basis}
For sufficiently large $\l$, the set $(s,\Phi_1s,...,\Phi_1^{n-1}s)$ is a local basis.
\end{lemma}
\begin{proof}
We know from Equation \eqref{Eq:matrix-Phi}, describing the form of $\Phi$, that there is a section $\tilde{s}$ of $L$ such that $(\tilde{s},\Phi_1\tilde{s},...,\Phi_1^{n-1}\tilde{s})$ is a local basis. Hence for large $\l$, $\tilde{B}=(\tilde{s},\nabla\tilde{s},...,\nabla^{n-1}\tilde{s})$ is still a basis since $\nabla=\l\Phi_1+\mathcal{O}(1)$. Normalizing $s=f\tilde{s}$, where $f$ is a nowhere vanishing function, to get a unit-volume basis, changes $\det \tilde{B}$ by $f^n$, proving the lemma.
\end{proof}

\begin{proof}[Proof of Proposition \ref{parametrisationlambda}]
The whole point is to analyze equations \eqref{paraboliccoord1} and \eqref{paraboliccoord2} in detail.
Let us start with $$\bar{\nabla}s = \bm\hat{\mu}_1(\l) s + \bm\hat{\mu}_2(\l) \nabla s + ... + \bm\hat{\mu}_n(\l) \nabla^{n-1}s.$$
Since $\bar{\nabla}s=(\l\Phi_2+\bar{\partial}+A_2+\l^{-1}\Psi_2)s$, the highest $\l$-term is $\l\Phi_2s=\l\mu_2\Phi_1s+...+\l\mu_n\Phi_1^{n-1}s$. On the other side, the highest term of $\nabla^ks$ is $\l^k\Phi_1^ks$ for $0\leq k \leq n-1$. Lemma \ref{Lemma:basis} shows that $(s, \Phi_1s, ..., \Phi_1^{n-1}s)$ is a basis for large $\l$. Hence, we can compare the highest terms and deduce that for $\l\rightarrow \infty$: $$\bm\hat{\mu}_k(\l) = \l^{2-k}\mu_k+\mathcal{O}(\l^{1-k}).$$ 

We can decompose $\nabla^ks$ and $\bar{\nabla}s$ in the basis $(s,\Phi_1s,...,\Phi_1^{n-1}s)$. We then get a linear system for $(\bm\hat{\mu}_k)_{1\leq k\leq n}$ with coefficients in $\C[\l]$. Hence by Cramer's rule, the $\bm\hat{\mu}_k$ are rational functions in $\l$. 
The same argument holds for $\bm\hat{t}_k$.

The last thing is to study the asymptotic behavior of $\bm\hat{t}_k$. For that, we have to study
$$\nabla^n s = \bm\hat{t}_{n}(\l)s + \bm\hat{t}_{n-1}(\l) \nabla s+...+\bm\hat{t}_{2}(\l)\nabla^{n-2}s.$$ 
The highest term of $\nabla^ns$ is not $\l^n \Phi_1^n$ since $\Phi_1^n=0$. The next term is given by $$\l^{n-1}\sum_{\ell=0}^{n-1}\Phi_1^\ell \circ(\partial+A_1)\circ\Phi_1^{n-1-\ell}s$$ where $\circ$ denotes the composition of differential operators.
On the other side, the highest terms are given by $\bm\hat{t}_k\l^{n-k}\Phi_1^{n-k}s$. When $\l$ goes to infinity, we compare coefficients in the basis $(s, \Phi_1s, ..., \Phi_1^{n-1}s)$ as before. Denote the dual basis by $(s^*,(\Phi_1 s)^*,...,(\Phi_1^{n-1}s)^*)$. We then compute
\begin{align*}
\l^{n-k}\bm\hat{t}_k &= \l^{n-1} (\Phi_1^{n-k}s)^*\left(\sum_{\ell=0}^{n-1}\Phi_1^\ell \circ(\partial+A_1)\circ\Phi_1^{n-1-\ell}s \right) \\
&= \l^{n-1}\sum_{\ell=0}^{n-k} (\Phi_1^{n-k-\ell}s)^*\left((\partial +A_1)\circ\Phi_1^{k-1}(\Phi_1^{n-k-\ell} s)\right) \\
&= \l^{n-1}\tr((\partial+A_1)\circ\Phi_1^{k-1}).
\end{align*}
Using the cyclic property of the trace, we get the expression for $t_k$ as stated in the proposition. 
\end{proof}

At the end of Subsection \ref{settingparab} we have noticed that $\bm\hat{t}_k$ and $\bm\hat{\mu}_k$ do not transform as tensors. We now show that the highest terms, $t_k$ and $\mu_k$, are tensors. Recall that $K=T^{*(1,0)}\Sigma$ is the canonical bundle and that we use the shorthand notation $K^i=K^{\otimes i}$.
\begin{prop}\label{highesttermtensor}
For $2\leq i\leq n$, we have $t_i \in \Gamma(K^i)$ and $\mu_i \in \Gamma(K^{1-i}\otimes \bar{K})$.
\end{prop}
\begin{proof}
Consider a holomorphic coordinate change $z\mapsto w(z)$. We compute how $\mu_i(z,\bar{z})$ and $t_i(z,\bar{z})$ change. 
For $\mu_i$, notice that $\Phi_1(w) = \Phi_1(z) \frac{dz}{dw}$, so using $$\Phi_2(z)d\bar{z} = \mu_3(z,\bar{z})\Phi_1^2(z)dz+...+\mu_n(z,\bar{z})\Phi_1^{n-1}(z)dz^{n-1},$$ we easily deduce $\mu_i(w,\bar{w})=\frac{d\bar{z}/d\bar{w}}{(dz/dw)^{i-1}}\mu_i(z,\bar{z})$ for all $i=2,...,n$.

For $t_i$, we use $t_i(z,\bar{z})=\tr(\Phi_1^{i-1}(\partial+A_1))$ from Proposition \ref{parametrisationlambda}. Since $\Phi_1$ and $\partial+A_1=(d_A)_{\partial}$ are both $(1,0)$-forms, we conclude that $t_i$ is a $(i,0)$-form, i.e. a smooth section of $K^i$.
\end{proof}
This proposition allows us to consider $(\mu_k, t_k)_{2\leq k\leq n}$ as global objects on the surface $\S$.

\subsection{$\mu$-holomorphicity from flatness}

We show that the $\mu$-holomorphicity term relating $(\mu_k, t_k)_{2\leq k\leq n}$ appears naturally in the curvature of $C(\l)$, as the leading term in $\l$. In the case of a flat family of connections, this allows us to identify $(\mu_k, t_k)_{2\leq k\leq n}$ with a representative of a point in $\cotang$ by Theorem \ref{conditionC}.
%More precisely, we exhibit the $\mu$-holomorphicity term as the highest term in $\l$ in the curvature of $C(\l)$. In particular, when the curvature of $C(\l)$ is zero, the highest term in $\l$ vanishes.

\medskip
As before, let $C(\l)$ be a family of $L$-parabolic connections of the form $C(\l)=\l\Phi+d_A+\l^{-1}\Psi$, where $\Phi$ is induced by a higher complex structure and $d_A$ is a connection satisfying $d_A(\Phi)=0$. We want to analyze the curvature of $C(\l)$.

Since the curvature is a local notion, we can work on an open chart $U\subset \Sigma$ and use the preferred representative \eqref{paragauge} in the $\mathcal{P}_L$-gauge class of $C(\l)$. Recall that in this gauge, the curvature of $C(\l)$ is described by $(\xi_k)_{2\leq k\leq n}$, concentrated in the first row, see Equation \eqref{Eq:parab-curvature}. All the $\xi_k$ now depend on $\l$. By Proposition \ref{parametrisationlambda}, we see that $\xi_k(\l)$ is a rational function in $\l$. We consider its Taylor expansion for $\l\to\infty$.

\begin{thm}\label{conditioncinconnection}
For $k\in\{2,...,n\}$, the Taylor expansion around $\l\to\infty$ of the parabolic curvature $\xi_k(\l)$ is given by $\xi_k(\l)=\l^{k-1}\zeta_k+\mathcal{O}(\l^{k-2})$, where $\zeta_k$ is the $\mu$-holomorphicity term for $t_k$ (see Equation \eqref{Eq:mu-holomorphic}). In particular, if $C(\l)$ is flat, the covector $(t_2,...,t_n)$ is $\mu$-holomorphic.
\end{thm}

The proof strategy is to observe that $\zeta_k$ does not involve quadratic (or even higher) terms in the $(\bm\hat{t}_k)_{2\leq k\leq n}$, nor terms with multiple derivatives. We then compute $\xi_k$ modulo these terms. Note that the set of these terms form a left-ideal in the space of differential operators. By abuse of language, we say that we compute $\xi_k$ modulo $\bm\hat{t}^2$ and modulo $\partial^2$ and write this as $\mod \bm\hat{t}^2,\partial^2$.

\begin{lemma}\label{curvaturemodmod}
The parabolic curvature $\xi_k(\l)$ satisfies
$$\xi_k(\l) \equiv (-\bar{\partial}\!+\!\bm\hat{\mu}_2\partial\!+\!k\partial\bm\hat{\mu}_k)\bm\hat{t}_k+\sum_{\ell=1}^{n-k}\left((\ell\!+\!k)\partial\bm\hat{\mu}_{\ell+2}+(\ell\!+\!1)\bm\hat{\mu}_{\ell+2}\del\right)\bm\hat{t}_{k+\ell} \mod \bm\hat{t}^2,\partial^2.$$
\end{lemma}

Note that this formula is almost the $\mu$-holomorphicity term. The proof is a direct and technical computation which we defer to Appendix \ref{appendix:C}.

\begin{lemma}\label{curvature-mod-derivatives}
The parabolic curvature $\xi_k(\l)$ vanishes modulo $\partial$ and $\bar\partial$. In other words, each term of $\xi_k(\l)$ contains derivatives.
\end{lemma}

\begin{proof}
Ignoring all derivatives collapses the differential operators to multiplication operations in the reduced Hilbert scheme. More precisely, the operator $\nabla$ reduces to $M_p$ and $\bar\nabla$ to $M_{\bar p}$, where $M_p$ and $M_{\bar{p}}$ denote the multiplication operators in $\C[p,\bar{p}]/I$.

For $1\leq k\leq n-1$, by Equation \eqref{Eq:xi-k}, we know that the curvatures $\xi_k$ are given by computing $[\nabla,\bar\nabla]\nabla^{k-1}s$. Modulo derivatives, the commutator $[\nabla,\bar\nabla]$ reduces to the commutator $[M_p,M_{\bar p}]$ which is zero.

For $\xi_n$, Equation \eqref{Eq:xi-n} shows that we have to compute $\nabla(C_n)-\bar\nabla(\nabla^n s)$. Modulo derivatives, this reduces to $M_p(\bar{p} p^{n-1})-M_{\bar p}(p^n) = 0$.
Therefore, there is no term left when we compute modulo $\partial$ and $\bar\partial$.
\end{proof}

\begin{proof}[Proof of Theorem \ref{conditioncinconnection}]
By definition of the parabolic curvature, we have 
\begin{equation}\label{Eq:xik-2}
\xi_k(\l)s=[\nabla,\bar\nabla]\nabla^{k-1}s,
\end{equation} 
where we put $\nabla=\partial+C_1(\l)$ and $\bar\nabla=\bar\partial+C_2(\l)$. This expression shows that $\xi_k(\l)$ is a polynomial in the $\bm\hat{t}_\ell$, $\bm\hat{\mu}_\ell$ and their derivatives. Hence to compute the highest term $\zeta_k$ in the Taylor expansion for $\l\to\infty$, we can replace $\bm\hat{t}_\ell$ by $\l^{\ell-1}t_\ell$ and $\bm\hat{\mu}_\ell$ by $\l^{2-\ell}\mu_\ell$ (see Proposition \ref{parametrisationlambda}).
Further, the curvature of $C(\l)$ has been computed in Equation \eqref{Eq:curv-3-term}, with result $[\nabla,\bar\nabla]=F(A)+[\Phi\wedge\Psi]+\mathcal{O}(\l^{-1})$. 
Since the highest term in $\l$ of $\nabla^{k-1}s$ is $\l^{k-1}\Phi_1^{k-1}s$, we see that $\xi_k(\l)=\l^{k-1}\zeta_k+\mathcal{O}(\l^{k-2})$.
In addition, since $F(A)+[\Phi\wedge\Psi]$ is a 2-tensor, it has to be of type $(1,1)$. Since $\Phi_1^{k-1}$ is of type $(k-1,0)$, we get that $\zeta_k$ is a tensor of type $(k,1)$.

We now prove that $\zeta_k$ does not involve quadratic terms in the $t_\ell$, nor terms with more than one derivative. The previous lemma then concludes.

The fact that $\zeta_k$ is a tensor of type $(k,1)$ implies that in each monomial term of $\zeta_k$, there is exactly one $\mu_\ell$ (for some $\ell$) and no $\bar\partial$-derivative, or exactly one $\bar\partial$-derivative and no $\mu_\ell$. Consider such a term $T$ of $\zeta_k$ and denote by $\alpha$ the number of $\partial$-derivatives and by $t_{\beta_1},...,t_{\beta_m}$ the $t_\ell$ (or their derivatives) appearing in $T$. 

If there is a $\bar\partial$-derivative in $T$, then from the tensor type we get $\alpha+\sum_i \beta_i=k$ and from the $\l$-contributions, we get $\sum_i \beta_i-m = k-1$. Hence $\alpha+m=1$. Since we do not have a $\mu_\ell$-term, we need $m\geq 1$, which then implies $\alpha=0$ and $m=1$. 

If $T$ contains a $\mu_\ell$, from the tensor type we get $\alpha+\sum_i\beta_i=k+\ell-1$ and from the $\l$-contributions, we get $\sum_i \beta_i-m=k+\ell-3$. Hence $\alpha+m=2$. From Lemma \ref{curvature-mod-derivatives} we see that $\alpha\geq 1$. If $\alpha=2$, then $m=0$ and we are left with $2=\alpha = k+\ell-1\geq 3$ (since $k$ and $\ell$ are at least 2), a contradiction.
Hence $\alpha=1=m$. 

As a conclusion, we see that in both cases, there is exactly one derivative and one $t_\ell$-term in each monomial contribution to $\zeta_k$.
Therefore, the highest term in $\xi_k(\l)$ is the same as in $\xi_k(\l) \mod \bm\hat{t}^2 \mod \partial^2$. Finally, the statement of the previous Lemma \ref{curvaturemodmod} concludes.
\end{proof}

The previous theorem allows us to identify $(\mu_k, t_k)_{2\leq k\leq n}$, which we extracted from the flat family $C(\l)$, with a point in $\cotang$. The $\mu_k$ are the higher Beltrami differentials coming from the higher complex structure, whereas the $t_k$ describe a $\mu$-holomorphic cotangent vector to that higher complex structure. 
%The question remains how to determine the coefficients of lower degree in $\bm\hat{\mu}_k$ and $\bm\hat{t}_k$. 

\subsection{Infinitesimal action and higher diffeomorphisms}\label{inf-action-2}

In Section \ref{symponconnections} we have described an action of the groupoid of pairs of line subbundles $(L,L')$ on the space of generic connections transforming $L$-parabolic connections into $L'$-parabolic connections. In particular, there is an action on the space of flat connections. 

Hence, to study this action, we consider a family of flat connections of the form $C(\l)=\l\Phi+d_A+\l^{-1}\Psi$ as in the previous section.
More precisely, we analyze the infinitesimal action on the highest terms $(\mu_k,t_k)_{2\leq k\leq n}$ and show that this action coincides with the action on higher complex structures via higher diffeomorphisms. 

\medskip
Consider an infinitesimal change of $L$ described by the change of the section $s$: 
\begin{equation}\label{Eq:action-with-param}
\delta s(\l) = \bm\hat{v}_1(\l)s+\bm\hat{v}_2(\l)\nabla s+...+\bm\hat{v}_{n}(\l)\nabla^{n-1}s=:\bm\hat{H}(\l)s.
\end{equation} 
As described in Section \ref{symponconnections}, this induces a special gauge transformation. Proposition \ref{variation-infinit} gives the change of the coordinates $\delta \bm\hat{\mu}_k$ and $\delta\bm\hat{t}_k$ under this action. 
%can be computed by $$\delta \bm\hat{Q}=[\bm\hat{H}, \bm\hat{Q}] \mod \bm\hat{I}$$
%where $\bm\hat{I}=\langle -\nabla^n+\bm\hat{t}_2\nabla^{n-2}+...+\bm\hat{t}_n, -\bar{\nabla}+\bm\hat{\mu}_1+\bm\hat{\mu}_2\nabla+...+\bm\hat{\mu}_n\nabla^{n-1} \rangle$ is a left-ideal in the space of differential operators.

Note that the variations $\bm\hat{v}_k$ also depend on $\l$. More precisely by Propositions \ref{computeX} and \ref{parametrisationlambda}, for $k\geq 2$ we know that $\bm\hat{v}_k(\l)$  is a rational function in $\l$ with highest term $\l^{2-k}v_k$ when $\l \rightarrow \infty$. 
The term $\bm\hat{v}_1$ is not a free parameter, but depends on the others. It assures that the trace of the infinitesimal gauge transform is zero. One can compute that $\bm\hat{v}_1(\l)$ has highest term of degree 0.

\begin{thm}\label{actionsymponlambdaconn}
Let $C(\l)$ be a family of flat connections as above. The action by the infinitesimal gauge, induced by Equation \eqref{Eq:action-with-param}, on the highest terms $(\mu_k,t_k)_{2\leq k\leq n}$ of the parameters $(\bm\hat{\mu}_k(\l),\bm\hat{t}_k(\l))_{2\leq k\leq n}$  of $C(\l)$, is the same as the infinitesimal action of a higher diffeomorphism generated by $H(p)=v_1+v_2p+...+v_np^{n-1}$ on the higher complex structure with cotangent vector described by $(\mu_k,t_k)_{2\leq k\leq n}$.
\end{thm}
The proof relies on the fact that the semi-classical limit of a commutator of differential operators is the Poisson bracket of the corresponding symbols. 

\begin{proof}
The variation of $\bm\hat{\mu}_k$ and $\bm\hat{t}_k$ are described in Proposition \ref{variation-infinit}. The variation of higher complex structures $\mu_k$ and their covectors $t_k$ are given by the classical counterpart by replacing the Lie bracket by the Poisson bracket, see Equation \eqref{idealvariation}. From these formulas, we see that they are linear in $H$, i.e. the variation of $H_1+H_2$ is the sum of the variations associated to $H_1$ and to $H_2$ separately. This allows to reduce to a monomial $\bm\hat{H}=\bm\hat{v}_k\nabla^{k-1}$.

We first compute the variation of $\bm\hat{\mu}_\ell$ under $\bm\hat{H}$. From Proposition \ref{variation-infinit}, we know that
\begin{equation}\label{Eq:proof.13}
\delta\bm\hat{\mu}_1 s+\delta\bm\hat{\mu}_2 \nabla s+...+\delta\bm\hat{\mu}_n \nabla^{n-1}s = [\bm\hat{v}_k\nabla^{k-1},-\bar\nabla+\bm\hat{\mu}_1+\bm\hat{\mu}_2\nabla+...+\bm\hat{\mu}_n\nabla^{n-1}]s.
\end{equation}
The highest term in $\lambda$ of the left hand side is 
\begin{equation}\label{Eq:aux.pif}
\lambda(\delta\mu_2 \Phi_1s+...+\delta\mu_n \Phi_1^{n-1}s),
\end{equation}
where we used $\bm\hat{\mu}_k=\l^{2-k}\mu_k+\mathcal{O}(\l^{1-k})$ for $k\geq 2$, $\bm\hat{\mu}_1=\mathcal{O}(1)$ and $\nabla^k s=\l^k\Phi_1^k s+\mathcal{O}(\l^{k-1})$. For the right hand side, note that
\begin{equation}\label{Eq:aux-higher-diff}
[\bar\nabla,\bm\hat{v}_k\nabla^{k-1}]s = \bar\partial \bm\hat{v}_k\nabla^{k-1} s = \l \bar\partial v_k \Phi_1^{k-1}s+\mathcal{O}(1),
\end{equation}
where we used $\bm\hat{v}_k=\l^{2-k}v_k+\mathcal{O}(\l^{1-k})$ and $[\bar\nabla,\nabla^{k-1}]s=0$, coming from the zero curvature condition $[\nabla,\bar\nabla]=0$. In addition we have
\begin{align}\label{Eq:aux-higher-diff-2}
[\bm\hat{v}_k\nabla^{k-1},\bm\hat{\mu}_\ell \nabla^{\ell-1}]s &= \bm\hat{v}_k\sum_{m=1}^{k-1}\binom{k-1}{m}\partial^m\bm\hat{\mu}_\ell \nabla^{k+\ell-m-2}s-\bm\hat{\mu}_\ell\sum_{m=1}^{\ell-1}\binom{\ell-1}{m}\partial^m \bm\hat{v}_k\nabla^{k+\ell-m-2}s  \nonumber \\
&= \l((k-1)v_k\partial\mu_\ell-(\ell-1)\mu_\ell\partial v_k)\Phi_1^{k+\ell-3}s+\mathcal{O}(1).
\end{align}
For $\ell=1$, this term is simply $\mathcal{O}(1)$, since $\bm\hat{\mu}_1=\mathcal{O}(1)$.
Using Equation \eqref{Eq:aux-higher-diff} and \eqref{Eq:aux-higher-diff-2}, we can compute the right hand side of \eqref{Eq:proof.13}. The result is
\begin{equation}\label{Eq:aux.pif2}
\l \bar\partial v_k\Phi_1^{k-1}s+\l\sum_{m=2}^n ((k-1)v_k\partial\mu_m-(m-1)\mu_m\partial v_k)\Phi_1^{k+m-3}s+\mathcal{O}(1).
\end{equation}
Since $(s,\Phi_1s,...,\Phi_1^{n-1}s)$ is a basis for large $\l$, we can compare coefficients in Equations \eqref{Eq:aux.pif} and \eqref{Eq:aux.pif2} to get the formula of $\delta\mu_\ell$:
$$\delta \mu_\ell = \left \{ \begin{array}{cl}
(\bar{\partial}-\mu_2\partial+(k-1)\partial\mu_2)v_k & \text{ if  } \ell=k \\
((k-1)\partial\mu_{\ell-k+2}-(\ell-k+2)\mu_{\ell-k+2}\partial)v_k &\text{ if  } \ell>k \\
0 & \text{ if  } \ell <k.
\end{array} \right.$$
This is identical with the formula for the variation of higher Beltrami differentials under $H=v_kp^{k-1}$ \cite[Proposition 3]{FockThomas}.

Second, we compute the variation of $\bm\hat{t}_\ell$ under $\bm\hat{H}$. From Proposition \ref{variation-infinit}, we know that
\begin{equation}\label{Eq:proof.14}
\delta\bm\hat{t}_2 \nabla^{n-2}s+...+\delta\bm\hat{t}_n s = [\bm\hat{v}_k\nabla^{k-1},-\nabla^n+\bm\hat{t}_2\nabla^{n-2}+...+\bm\hat{t}_n]s.
\end{equation}
The highest term in $\l$ of the left hand side is $\l^{n-1}(\delta t_2\Phi_1^{n-2}s+...+\delta t_n s)$. For the right hand side, we can use Equation \eqref{Eq:aux-higher-diff-2} to get
$$[\bm\hat{v}_k\nabla^{k-1}, \bm\hat{t}_\ell\nabla^{n-\ell}]s = \l^{n-1}((k-1)v_k\partial t_\ell-(n-\ell)t_\ell\partial v_k)\Phi_1^{n+k-\ell-2}s+\mathcal{O}(\l^{n-2}).$$
Finally, we compute the last term for $k> 2$:
\begin{align*}
    [\nabla^n,\bm\hat{v}_k\nabla^{k-1}]s &= n\partial \bm\hat{v}_k\nabla^{n+k-2}s+\sum_{\ell=2}^n \binom{n}{\ell}\partial^\ell\bm\hat{v}_k \nabla^{n+k-\ell-1}s \\
&= n\partial \bm\hat{v}_k\nabla^{k-2}(\bm\hat{t}_2\nabla^{n-2}s+...+\bm\hat{t}_n s)+\sum_{\ell=2}^n \binom{n}{\ell}\partial^\ell\bm\hat{v}_k \nabla^{n+k-\ell-1}s \\
&= \l^{n-1} n\partial v_k(t_2\Phi_1^{n+k-4}s+...+t_n\Phi_1^{k-2}s)+\mathcal{O}(\l^{n-2}).
\end{align*}
In the last line, we used that $\binom{n}{\ell}\partial^\ell\bm\hat{v}_k \nabla^{n+k-\ell-1}s=\mathcal{O}(\l^{n+1-\ell})$. For $k=2$, we get the same formula with an extra term $\binom{n}{2}\partial^2 v_2\Phi_1^{n-1}s$.
This term will not influence the variation $\delta\bm\hat{t}_\ell$ since there is no $\Phi_1^{n-1}s$ in the highest term on the left hand side of \eqref{Eq:proof.14}.
Again, since $(s,\Phi_1s,...,\Phi_1^{n-1}s)$ is a basis, comparing coefficients gives the result:
$$\delta t_\ell = (k-1)v_k\partial t_{k+\ell-2}+(k+\ell-2)t_{k+\ell-2}\partial v_k,$$
where we put $t_\ell=0$ for $\ell>n$. This is the infinitesimal variation of a cotangent vector of a higher complex structure, see \cite[Proposition 7.18]{Fock-bundles}.
\end{proof}

%\begin{Remark}
%We see that a term $\bm\hat{v}_k\nabla^{k-1}$ can influence $\bm\hat{\mu}_i$ with $i<k$ (unlike the case higher complex structures where $H$ acts like $H \mod I$), but it does not influence the highest term $\mu_i$. In the same vein, a term $\bm\hat{v}_k\nabla^{k-1}$ with $k>n$ acts on parabolic connections, but not on the highest terms.
%\end{Remark}

\section{Reality constraint and Toda systems}\label{finalstep}

In this final section, we impose a reality constraint on the family of $L$-parabolic connections studied before. We show that this forces the family to be flat. We study the flatness equations for small $n$, and find Toda integrable systems in special cases. 
Our main conjecture states that these families of flat connections are parametrized by higher complex structures and $\mu$-holomorphic cotangent vectors.

\subsection{Reality constraint and main conjecture}

Consider a higher complex structure $I$ on a closed surface $S$ inducing a complex vector bundle $V$ with a holomorphic line subbundle $L$ isomorphic to $K^{(n-1)/2}$. We equip $V$ with a volume form and a Hermitian structure $h$ giving a Lie algebra involution $\rho$ on $\mathfrak{sl}(V)$ defined by $\rho(M)=-M^{*_h}$.

\begin{definition}
    A Hermitian structure $h$ is called $\mathbf{L}$\textbf{-compatible} if $\mathfrak{p}_L^*\cap\mathfrak{sl}(V)^{-\rho}=\{0\}$, where $\mathfrak{p}_L^*$ denotes the trace dual to the parabolic subalgebra $\mathfrak{p}_L$ determined by $L$, and $\mathfrak{sl}(V)^{-\rho}$ denotes the set of Hermitian matrices, i.e. satisfying $M^{*_h}=M$.
\end{definition}

In the sequel, we consider $L$-compatible Hermitian structures. They form an open set inside all Hermitian structures.
\begin{example}
    For $\mathfrak{sl}_2$ and $\mathfrak{p}_L=\left(\begin{smallmatrix} *&*\\0&*\end{smallmatrix}\right)$, the trace dual is given by $\mathfrak{p}_L^*=\left(\begin{smallmatrix} 0&*\\0&0\end{smallmatrix}\right)$. Let us determine the Hermitian structures which are not $L$-compatible. Consider $h=\left(\begin{smallmatrix} h_1&h_2\\h_3&h_4\end{smallmatrix}\right)$ such that $\det(h)=1$ and $hh^\dagger=\mathrm{id}$. The involution is then given by $\rho(M)=hM^\dagger h^{-1}$. We look for a non-zero matrix $M\in\mathfrak{p}_L^*$ satsifying $M^{*_h}=M$. For $M=\left(\begin{smallmatrix} 0&a\\0&0\end{smallmatrix}\right)$ we get
$$M^{*_h}=\begin{pmatrix} h_1&h_2\\h_3&h_4\end{pmatrix}\begin{pmatrix} 0&0\\\bar{a}&0\end{pmatrix}\begin{pmatrix} h_4&-h_2\\-h_3&h_1\end{pmatrix}=\begin{pmatrix} h_2h_4\bar{a}&-h_2^2\bar{a}\\h_4^2\bar{a}&-h_2h_4\bar{a}\end{pmatrix}.$$
This is again in $\mathfrak{p}_L^*$ if and only if $h_4=0$. By $\det(h)=1$ and $hh^\dagger=\mathrm{id}$, this implies that $h=\left(\begin{smallmatrix} 0&e^{i\theta}\\-e^{-i\theta}&0\end{smallmatrix}\right)$ for some $\theta\in \R$. Note that this set of Hermitian structures which are not $L$-parabolic is of codimension 2 in the space of all Hermitian structures.
\end{example}

We consider a family of $L$-parabolic connections with parameter $\l\in\C^*$ of the form
\begin{equation}\label{Eq:c-lambda-3}
C(\l)=\l\Phi + d_A + \l^{-1}\Psi,
\end{equation}
where $\Phi$ is the field on $V$ induced by the higher complex structure $I$, $d_A$ is a connection and $\Psi\in\Omega^1(S,\mathfrak{sl}(V))$. Using the Hermitian structure $h$, we now impose a reality constraint:
\begin{equation}\label{Eq:reality-constraint}
-C(-1/\bar{\l})^{*_h}=C(\l).
\end{equation}
This implies that $d_A$ is $h$-unitary and $\Psi=\Phi^{*_h}$ is the Hermitian conjugate of $\Phi$. This kind of reality constraints is typical in the twistor approach of hyperkähler manifolds \cite[Section 3 (F)]{HKLR}.

The reality constraint implies that $C(\l)$ is a family of \emph{flat} connections. This holds true in a more general setting:

\begin{prop}\label{Prop:flat-para-equiv}
    Consider a family of connections $M(\l)=\sum_{i=-N}^N A_i\l^i$, where $A_0$ is a connection on $V$ and $A_i\in\Omega^1(S,\mathfrak{sl}(V))$ for all other $i$. Consider also an $L$-compatible Hermitian structure $h$ on $V$. Finally, we impose $M(\l)=-M(-1/\bar{\l})^{*_h}$. Then $M(\l)$ is flat for all $\l\in\C^*$ if and only if it is $L$-parabolic for all $\l\in\C^*$.
\end{prop}

\begin{proof}
    Since a flat connection is always $L$-parabolic, we only have to prove one direction. Consider $M(\l)$ as in the proposition and assume it to be a family of $L$-parabolic connections. We show that it is flat.
    
    The condition $M(\l)=-M(-1/\bar{\l})^{*_h}$ implies that $A_0$ is $h$-unitary and for all $k>0$, we get $A_{-k}=(-1)^{k+1}A_k^{*_h}$. The curvature of $M(\l)$ is given by
    \begin{equation}\label{Eq:curv-general}
    F(M(\l)) = dM(\l)+M(\l)\wedge M(\l) = \textstyle\sum_k \left(dA_k+\textstyle\sum_\ell [A_\ell,A_{k-\ell}]\right)\l^k.
    \end{equation}
    Put $M_k=dA_k+\sum_\ell [A_\ell,A_{k-\ell}]$. By assumption, all $M_k$ are in the trace-dual of $\mathfrak{p}_L$. Since
    $$dA_{-k}+\textstyle\sum_\ell [A_{-\ell},A_{-k+\ell}] = (-1)^{k+1}\left(dA_k+\textstyle\sum_\ell [A_\ell,A_{k-\ell}]\right)^{*_h},$$
    we see that $(M_k)^{*_h}$ is also in $\mathfrak{p}_L^*$. Hence $M_k+M_k^{*_h}\in\mathfrak{p}_L^*\cap \mathfrak{sl}(V)^{-\rho}=\{0\}$. Therefore $iM_k\in \mathfrak{p}_L^*\cap \mathfrak{sl}(V)^{-\rho}=\{0\}$, i.e. $M_k=0$. This holds for all $k$, so the connection $M(\l)$ is flat.
\end{proof}

We come back to our family of connections $C(\l)=\l\Phi + d_A + \l^{-1}\Phi^{*_h}$. By the previous section (in particular Proposition \ref{parametrisationlambda} and Theorem \ref{conditioncinconnection}), we know that $C(\l)$ contains tensorial parameters $(\mu_k,t_k)_{2\leq k\leq n}$. The $\mu_k$ are the higher Beltrami differentials. The $t_k$ are locally given by $t_k=\tr(\Phi_1^{k-1}(\partial+A_1))$ and are $\mu$-holomorphic.
Our main conjecture is that these parameters completely describe the family of flat connections $C(\l)$:

\begin{conj}\label{nahc}
For $(\mu_k,t_k)_{2\leq k\leq n}\in T^*\mathbb{M}^n(\S)$, where the $t_k$ are small and $\mu$-holomorphic, there is a unique (up to constant scale) Hermitian structure $h$ on the bundle $V$ on $\S$ and a unique (up to unitary gauge) family of flat connections $C(\l)=\l\Phi+d_A+\l^{-1}\Phi^{*_h}$ satisfying
\begin{enumerate}
\item $\Phi$ is induced by the higher complex structure $(\mu_2,...,\mu_n)$,
\item $C(\l)$ satisfies the reality condition \eqref{Eq:reality-constraint},
\item $t_k=\tr(\Phi_1^{k-1}(\partial+A_1))$ for all $k\in\{2,...,n\}$.
\end{enumerate}
\end{conj}

From Proposition \ref{Prop:para-L-para-conns}, we know that $C(\l)$ is described by $(\bm\hat{t}(\l),\bm\hat{\mu}_k(\l))_{2\leq k\leq n}$. Conjecture \ref{nahc} says that these parameters should be determined by their highest terms.

\begin{Remark}
A step towards establishing this conjecture is described in \cite{Fock-bundles}, introducing the notion of a \emph{Fock bundle}. There, the connection $d_A$ is determined from the data of $\Phi$, $h$ and the covectors $(t_2,...,t_n)$. 
\end{Remark}

In the rest of the paper, we analyze Conjecture \ref{nahc} in some special situations. Since we know that $C(\l)$ is a family of flat connections, we have to solve $F(C(\l))=0$. By Equation \eqref{Eq:curv-general}, this amounts to
\begin{equation}\label{Eq:flatness.2}
d_A(\Phi)=0\;\; \text{ and } \;\; F(A)+[\Phi\wedge\Phi^{*_h}]=0.
\end{equation}
The first equation should be thought of as an equation on the connection $d_A$ (which is a linear equation), while the second one is a non-linear PDE in the Hermitian structure $h$. Note the similarity of the second equation with the Hitchin equation in the nonabelian Hodge theory. The solution scheme in that theory is based on harmonic analysis \cite{Corlette, Simpson} and uses in a crucial way a fixed complex structure on the surface $S$ and the holomorphic structure on the bundle $V$. In our case, the bundle $V$ has no holomorphic structure, so the techniques do not apply. Only in the case of trivial higher complex structure ($\mu_k=0$ for all $k$) with trivial covector ($t_k=0$ for all $k$), we are in the setting of the non-abelian Hodge theory, where the Higgs field is the uniformizing one.

%We will see that for zero covector and trivial higher complex structure, we are in the setting of the non-abelian Hodge correspondence with principal nilpotent Higgs field. Solving the flatness equations is then equivalent to the Toda integrable system. Keeping trivial higher complex structure, but non-trivial covector, we find a generalization of the Toda system.

%\subsection{Toda integrable system}

%\textcolor{red}{Some general remarks on the Toda system, different versions etc.}

\subsection{Local standard form}\label{standard-form}

Before studying the flatness equations, we can reduce the family of connections $C(\l)=\l\Phi + d_A + \l^{-1}\Phi^{*_h}$ to a local standard form. Consider an open chart $U\subset S$ with coordinates $(z,\bar{z})$ and decompose $\Phi=\Phi_1dz+\Phi_2d\bar{z}$ and $d_A=d+A_1dz+A_2d\bar{z}$. Fix a trivialization of $V\!\mid_U$ such that the Hermitian structure is the standard one, i.e. $M^{*_h}=M^\dagger$ is the conjugate transpose.

\begin{lemma}\label{phi1lower}
There is a unitary gauge transformation such that $\Phi_1$ becomes lower triangular with entries of coordinates $(i+1,i)$ given by positive real functions of the form $e^{\varphi_i}$ for all $i=1,...,n-1$.
\end{lemma}
\begin{proof}
It is sufficient to prove the lemma at a point $z\in S$.
Since gauge transformations act by conjugation on $\Phi_1(z)$, which is nilpotent, there is an invertible matrix $G(z)\in \GL_n(\C)$ such that $G\Phi_1G^{-1}$ is strictly lower triangular. Since $\Phi(z)$ varies smoothly with $z$, so does $G(z)$. We omit the dependence in $z$ in the rest of the proof.

We decompose $G$ as $G=TU$ where $T$ is lower triangular (not strict) and $U$ is unitary (Gram-Schmidt). Then the matrix $U\Phi_1U^{-1} = T^{-1}(G\Phi_1G^{-1})T$ is already lower triangular. So we have conjugated $\Phi_1$ to a lower triangular matrix via a unitary gauge.

Finally, we use a diagonal unitary gauge to change the arguments of the matrix elements with coordinates $(i+1,i)$ to zero. Since $\Phi_1(z)$ is principal nilpotent, all these elements are non-zero, so strictly positive real numbers which can be written as $e^{\varphi_i(z)}$ with $\varphi_i(z) \in \R$.
\end{proof}

Notice that the unitary gauge preserves the Hermitian conjugation, so the reality condition \eqref{Eq:reality-constraint} is preserved. In addition, the formula \eqref{Eq:t_k-form} for $t_k$ simplifies to $t_k=\tr(\Phi_1^{k-1}A_1)$ in this gauge.

Another simplification happens whenever the higher complex structure is trivial, i.e. when $\mu_k=0$ for all $k=2,...,n$. This is equivalent to $\Phi_2=0$. 

\begin{lemma}\label{a1upper}
For $\Phi_2=0$ (trivial higher complex structure) and $\Phi_1$ lower triangular, the flatness of $C(\l)$ implies that $A_1$ is upper triangular.
\end{lemma}
\begin{proof}
We write $A_1=A_l + A_u$ where $A_l$ and $A_u$ are respectively the strictly lower and the (not strictly) upper part of $A_1$. Thus we have $A_2 = -A_l^{\dagger}-A_u^{\dagger}$.
The flatness condition $d_A(\Phi)=0$ then gives $$0=\bar{\partial}\Phi_1 + [\Phi_1,A_u^{\dagger}]+[\Phi_1,A_l^{\dagger}].$$
Since the first two terms are lower triangular (the conjugate transpose exchanges upper and lower triangular matrices), so is the third term $[\Phi_1,A_l^{\dagger}]$.

A simple computation shows that a commutator between a principal nilpotent lower triangular matrix and a non-zero strictly upper triangular matrix can never be strictly lower triangular. Thus, $A_l=0$.
\end{proof}

\subsection{Usual complex structures}\label{Sec:n2}

Consider the smallest case where $n=2$, so the higher complex structure is a usual complex structure. Consider a local chart with coordinates $(z,\bar{z})$, induced from the complex structure. We can use the standard form for $\Phi_1$ and since the Beltrami differential vanishes in the induced complex structure, Lemma \ref{a1upper} gives that $A_1$ is upper triangular.

Therefore, we can write $\Phi_1 = \left(\begin{smallmatrix} 0 & 0 \\ e^{\varphi} & 0\end{smallmatrix}\right)$, $A_1 = \left(\begin{smallmatrix} a_0 & a_1 \\ 0 & -a_0\end{smallmatrix}\right)$ and $A_2 = -A_1^{\dagger}$. So we have 
$$C(\l)=d+\begin{pmatrix} a_0 & a_1 \\ \l e^{\varphi} & -a_0\end{pmatrix}dz+ \begin{pmatrix} -\bar{a}_0 & \l^{-1}e^{\varphi} \\ -\bar{a}_1 & \bar{a}_0\end{pmatrix}d\bar{z}.$$
Notice that this is Example \ref{examplen2} with $\mu_2=0$ and $b_1=e^{\varphi}$. The flatness condition gives 
$$ \left \{\begin{array}{cl}
\bar{\partial}\varphi &= \; -2\bar{a}_0 \\
\bar{\partial}a_1 &=\;  2\bar{a}_0a_1 \\
\partial \bar{a}_0+\bar{\partial}a_0 &= \; -a_1\bar{a}_1-e^{2\varphi}.
\end{array}\right. $$
The first equation gives $a_0=-\frac{\del \varphi}{2}$. For the second, write $a_1=t_2e^{-\varphi}$, where $t_2=\tr \Phi_1A_1$ is a quadratic differential. Then the second equation is equivalent to $t_2$ being holomorphic. Finally, the last equation gives $$\partial\bar{\partial} \varphi = e^{2\varphi} + t_2\bar{t}_2e^{-2\varphi},$$ which is the so-called $\boldsymbol{\cosh}$\textbf{-Gordon equation}. For small $t_2$, this equation is elliptic, and allows a unique global solution. Therefore, we see that $C(\l)$ is uniquely determined by a complex structure and a (small) holomorphic quadratic differential $t_2$. More details for this case can be found in \cite{Fock}, in particular a link to minimal surface sections in $\S\times \R$.

Let us put $C(\l)$ in $L$-parabolic gauge. The gauge transformation is given by $P=\left(\begin{smallmatrix} p_1& p_2 \\ 0& 1/p_1 \end{smallmatrix}\right)$ with $p_1=(\l e^{\varphi})^{1/2}$ and $p_2=\partial\varphi/p_1$. Using Example \ref{Ex:n2}, we then get
$$C(\l)\sim d+\begin{pmatrix}0&\bm\hat{t}_2(\l) \\1&0\end{pmatrix}dz+\begin{pmatrix}-\tfrac{1}{2}\partial\bm\hat{\mu}_2&-\tfrac{1}{2}\partial^2\bm\hat{\mu}_2+\bm\hat{t}_2\bm\hat{\mu}_2\\\bm\hat{\mu}_2(\l)&\tfrac{1}{2}\partial\bm\hat{\mu}_2\end{pmatrix}d\bar{z},$$
with $\bm\hat{t}_2(\l)=\l t_2+(\partial\varphi)^2-\partial^2\varphi$ and $\bm\hat{\mu}_2(\l)=-\l^{-1}\bar{t}_2e^{-2\varphi}$.

\subsection{$\mathrm{SL}_3(\C)$-connections and trivial higher complex structures}\label{n2n3}

Let us study the next smallest rank, $n=3$, with trivial $3$-complex structure on an open chart $U\subset \S$. We use the local standard form from Subsection \ref{standard-form}. This allows us to write 
$$C(\l)=\begin{pmatrix} a_0 & b_0 & c_0 \\ \l c_1 & a_1 & b_1 \\ \l b_2 & \l c_2 & a_2 \end{pmatrix}dz+ \begin{pmatrix} -\bar{a}_0 & \l^{-1}\bar{c}_1 & \l^{-1}\bar{b}_2 \\ -\bar{b}_0 & -\bar{a}_1 & \l^{-1}\bar{c}_2 \\ -\bar{c}_0 & -\bar{b}_1 & -\bar{a}_2\end{pmatrix}d\bar{z},$$
with $c_1=e^{\varphi_1}, c_2=e^{\varphi_2} \in \R_+$. Further, the expressions for the holomorphic differentials are $t_3=\tr \Phi_1^2A_1 = c_0c_1c_2$ and $t_2= \tr\Phi_1A_1 = b_0c_1+b_1c_2+b_2c_0$, hence $c_0=t_3e^{-\varphi_1-\varphi_2}$ and $b_1=-e^{\varphi_1-\varphi_2}b_0-b_2t_3e^{-2\varphi_2-\varphi_1}$.
The flatness condition and the zero trace condition then give $a_0=-\frac{2}{3}\partial\varphi_1-\frac{1}{3}\partial\varphi_2$, $a_1=\frac{1}{3}\partial\varphi_1-\frac{1}{3}\partial\varphi_2$ and $a_2=-a_0-a_1$.

\medskip
First, let us consider the case where $t_2=t_3=0$. Then $c_0=0$ and $b_1=-e^{\varphi_1-\varphi_2}b_0$. The remaining equations of the flatness are
$$ \left \{\begin{array}{cl}
\bar{\partial}b_2 &= \; b_2(\bar{\partial}\varphi_1+\bar{\partial}\varphi_2)-\bar{b}_0(e^{\varphi_2}+e^{2\varphi_1-\varphi_2}) \\
-\bar{\partial}b_0 &= \;  b_0\bar{\partial}\varphi_1+\bar{b}_2e^{\varphi_2} \\
2\partial\bar{\partial}\varphi_1 &= \;  2e^{2\varphi_1}-e^{2\varphi_2}+b_2\bar{b}_2+b_0\bar{b}_0(2-e^{2\varphi_1-2\varphi_2}) \\
2\partial\bar{\partial}\varphi_2 &= \;  2e^{2\varphi_2}-e^{2\varphi_1}+b_2\bar{b}_2+b_0\bar{b}_0(-1+2e^{2\varphi_1-2\varphi_2}). 
\end{array}\right. $$

We see that $b_0=b_2=0$ solves the two first equations.
The system then reduces to the \textbf{Toda integrable system} for $\mf{sl}_3$. 

The Toda system associated to a Cartan matrix $(C_{ij})$ of size $n\times n$, is a system of differential equations for the unknown functions $(\varphi_1,...,\varphi_n)$ on $\C$ given by $$\partial\bar\partial \varphi_i=\sum_j C_{ij}e^{\varphi_j}.$$
For the Cartan matrix of type $A_1$, we simply get the Liouville equation $\partial\bar\partial\varphi=2e^\varphi$. Toda integrable systems have been intensely studied, in particular their link to flat connections, see for example \cite{Leznov} and the references therein. Denote by $(e_\alpha,h_\alpha,f_\alpha)$ the standard Cartan--Weyl basis of the Lie algebra $\mathfrak{g}$ associated to the Cartan matrix $(C_{ij})$. Then the Toda equations are equivalent to the flatness of the connection 
\begin{equation}\label{Eq:tridiagonal}
d+\left(\textstyle\sum_\alpha (a^\alpha e_\alpha +b^\alpha h_\alpha)\right)dz+\left(\textstyle\sum_\alpha (\bar{a}^\alpha f_\alpha +\bar{b}^\alpha h_\alpha)\right)d\bar{z},
\end{equation}
with $a^\alpha \bar{a}^\alpha = e^{\varphi_\alpha}$. Note that the connection matrix is ``tridiagonal'', meaning that the only non-zero entries are on the main diagonal and just above and just below. For a global meaning of these equations on a Riemann surface, see \cite[Section 5.3]{Baraglia}.

Coming back to our family of flat connections, the solution to the Toda equations uniquely describes $C(\l)$. Thus, we get the same solution as the one obtained from the non-abelian Hodge correspondence applied to the uniformizing Higgs field \cite{AF, Baraglia}.
%For invariance under the real form $\tau$, we need $\varphi_1=\varphi_2$ ($b_0$ and $b_2$ can be arbitrary).

%Note that this implies that $d_A$ is diagonal, which is in accordance to the non-abelian Hodge theory, where $d_A$ is the Chern connection.

\medskip
Second, let us consider the case for $t_2=0$ and $t_3\neq 0$. If we impose $\varphi_1=\varphi_2=\varphi$ and $b_0=b_1=b_2=0$, we get $A_1=\left(\begin{smallmatrix}\partial\varphi &0&t_3e^{-2\varphi}\\0&0&0\\ 0&0&-\partial\varphi\end{smallmatrix}\right)$. The flatness equations then reduce to $\bar\partial t_3=0$ and the \textbf{\c{T}i\c{t}eica equation} 
\begin{equation}\label{Titeica}
2\del\delbar \varphi = e^{2\varphi}+t_3\bar{t}_3e^{-4\varphi}.
\end{equation}
From \cite{Loftin} and \cite{Labourie-cubic}, we know that \c{T}i\c{t}eica's equation is linked to affine spheres and the $\mathrm{SL}_3(\R)$-Hitchin component. 

\begin{Remark}
    We are not here in the setting of cyclic Higgs bundles as described in \cite{Baraglia}. For cyclic Higgs bundles, the connection $d_A$ is the Chern connection and is diagonal. In our case since $t_3=c_0c_1c_2\neq 0$, we see that $d_A$ is not diagonal.
\end{Remark}

Finally for general $t_2$ and $t_3$, the equation $d_A(\Phi)=0$ allows to compute the connection matrix $A_1$ as function of $b_2, \varphi_1, \varphi_2, t_2, t_3$. The off-diagonal terms of $F(A)+[\Phi\wedge\Phi^\dagger]=0$ then give a differential equation for $b_2$, and $\bar\partial t_2=\bar\partial t_3=0$. The diagonal terms give
$$ \left \{\begin{array}{cl}
2\partial\bar{\partial}\varphi_1 &= e^{2\varphi_2}-2e^{\varphi_1}-t_3\bar{t}_3e^{-2\varphi_1-2\varphi_2}+\lvert a_5\rvert^2-2\lvert a_1\rvert^2-\lvert b_2\rvert^2 \\
2\partial\bar{\partial}\varphi_2 &= e^{2\varphi_1}-2e^{\varphi_2}-t_3\bar{t}_3e^{-2\varphi_1-2\varphi_2}+\lvert a_1\rvert^2-2\lvert a_5\rvert^2-\lvert b_2\rvert^2
\end{array}\right. $$
This can be considered as a generalized Toda system, which we recover for $t_2=t_3=0$.
%Before going to the general case, we push the similarity to Higgs bundles further by choosing a special gauge.

\subsection{Trivial higher complex structure and zero covector}\label{flatconnectionlambda}

We return to the study of the flat family $C(\l)$ for the general Lie algebra $\mathfrak{sl}_n$. For the trivial higher complex structure with zero covector, we find the following result, generalizing the observations for $n=2$ and $n=3$ from the previous subsections.

\begin{prop}\label{linktohiggs}
A solution to the $\mathfrak{sl}_n$-Toda integrable system determines a flat family $C(\l)$ satisfying Equations \eqref{Eq:c-lambda-3}, \eqref{Eq:reality-constraint}, $\Phi_2=0$ and $t_k=0$ for all $k$. This family $C(\l)$ is the same as the twistor family constructed by the non-abelian Hodge correspondence applied to the uniformizing Higgs field.
\end{prop}

\begin{Remark}
The recent progress in \cite{Fock-bundles} (in particular Section 3.2) shows the equivalence of $C(\l)$ with $\Phi_2=0$ and $t_k=0$ for all $k$, and the uniformizing Higgs bundles.
\end{Remark}

\begin{proof}
Using Lemma \ref{phi1lower} and \ref{a1upper}, we can write $C_1(\l)$ in the following form: $$C_1(\l)=a_0+a_1T+...+a_nT^n$$ where $a_i$ are diagonal matrices and $T$ is given by 
\begin{equation}\label{matrixT}
T=\begin{pmatrix} 0& 1 & & \\ &\ddots&\ddots & \\ &&0& 1 \\ \l &&&0 \end{pmatrix}.
\end{equation}
We denote by $a_{i,j}$ the $j$-th entry of the diagonal matrix $a_i$ and by $a_i'$ the shifted matrix defined by $(a')_{i,j} = a_{i,j+1}$. We write $a^{(k)}$ for the shift applied $k$ times. Notice that $aT=Ta^{(n-1)}$. We can then write 
$$C_2(\l)= a_0^*+T^{-1}a_1^*+...+T^{-n}a_n^*$$ where $a^*_{i,j}=\pm \bar{a}_{i,j}$, the sign depends on whether the coefficient comes with a $\l$ or not in $C_2(\l)$.

By the standard form (Lemma \ref{phi1lower}) we can further impose $a_{n,i}=e^{\varphi_i}$ for $i=1,...,n-1$ and $a_{n,0}=0$ since $0=t_n=\prod_i a_{n,i}$. 
One of the flatness equations gives $\delbar a_n = a_n (a_0^{(n-1)}-a_0)$. Together with the condition that the trace is 0, we can compute $a_0$. We get
\begin{equation}\label{a0i}
a_{0,i}= \sum_{k=1}^{i-1}\frac{k}{n}\del\varphi_k-\sum_{k=i}^{n-1}\frac{n-k}{n}\del\varphi_k.
\end{equation}
The other equations give a system of differential equations in $a_1, ..., a_{n-1}$ which is quadratic. It allows the solution $a_i=0$ for all $i=1,...,n-1$. 
In that case, using a diagonal gauge $\diag (1,\lambda, ..., \lambda^{n-1})$ the connection $C(\l)$ becomes independent of $\l$:
\begin{equation}\label{mu0t0}
C(\l)=d+\begin{pmatrix} a_{0,1}&&& \\ e^{\varphi_1} &a_{0,2} && \\ & \ddots &\ddots & \\ && e^{\varphi_{n-1}} & a_{0,n} \end{pmatrix}dz+\begin{pmatrix}-\bar{a}_{0,1} & e^{\varphi_1} && \\ &-\bar{a}_{0,2} & \ddots & \\ &&\ddots & e^{\varphi_{n-1}} \\ &&&-\bar{a}_{0,n} \end{pmatrix}d\bar{z},
\end{equation}
where all empty spots are zero.
This is precisely the form of the Toda system. It is known that the Hitchin equations for the uniformizing Higgs field are the Toda equations for $\mf{sl}_n$, see for example \cite[Proposition 3.1]{AF}.
\end{proof}
Notice that in particular the gauge class of the connection $C(\l)$ is independent of $\l\in \C^*$ (i.e. we have a variation of Hodge structure). 

Putting \eqref{mu0t0} in $L$-parabolic gauge, we get the following explicit formula for our parameters $\bm\hat{t}(\l)$ and $\bm\hat{\mu}(\l)$:
\begin{prop}
For $\Phi_2=0$ and $t_k=0$ for all $k=2,...,n$, we get $\bm\hat{\mu}_k(\l)=0$ and $\bm\hat{t}_k(\l)=w_k$ (independent of $\lambda$), where the $w_k$ are given by the so-called \emph{Miura transform} $\del^n-w_2\del^{n-2}-...-w_n = (\partial-2a_{0,n})(\partial-2a_{0,n-1})\cdots(\partial-2a_{0,1})$. Furthermore, $A_1$ is diagonal given by Equation \eqref{a0i} and the $L$-parabolic gauge is upper triangular.
\end{prop}

The proof is a combination of known results. The fact that $\bm\hat{\mu}_k(\l)=0$ and that $A_1$ is diagonal is Theorem 4.1 in \cite{AF}. The $L$-parabolic gauge is upper triangular follows from Proposition 4.4 in \cite{AF}. The formula for $\bm\hat{t}_k$, the Miura transform, can be found in Sections 4.1 and 4.2 of \cite{Bonora}.

For $n=2$, the Miura transform gives $\partial^2-\bm\hat{t}_2=(\partial-\partial\varphi)(\partial+\partial\varphi)=\partial^2+\partial^2\varphi-(\partial\varphi)^2$. Hence $\bm\hat{t}_2 = (\partial\varphi)^2-\partial^2\varphi$ which is in accordance with the result at the end of Section \ref{Sec:n2}.

\subsection{Generalized Toda system}

We indicate how to solve locally the flatness equations of $C(\l)$, leading to a generalized Toda system. We give a framework for this system in terms of the loop algebra $\mathcal{L}(\mathfrak{sl}_n)$. For loop algebras, we recommend \cite{Khesin-Wendt}.

On an open chart $U\subset S$, we can use Lemma \ref{phi1lower} to put $\Phi_1$ into the standard form (lower triangular with entries $(e^{\varphi_i})_{1\leq i\leq n-1}$ just below the main diagonal). From the analysis of the cases for $n=2$ and $n=3$, we expect the following:
\begin{enumerate}
    \item The equation $d_A(\Phi)=0$, together with the differentials $t_k=\tr \Phi_1^{k-1}(\partial+A_1)$ for $2\leq k\leq n$, determine the connection matrix $A_1$ as a function of the entries of $\Phi$ and the $t_k$.
    \item The off-diagonal terms of $F(A)+[\Phi\wedge\Phi^\dagger]=0$ determine all entries of $\Phi$ as functions of $(\mu_k,t_k,\varphi_{k-1})_{2\leq k\leq n}$, and show that the $t_k$ are $\mu$-holomorphic for all $k=2,...,n$.
    \item The diagonal terms of $F(A)+[\Phi\wedge\Phi^\dagger]=0$ give a system of differential equations for $(\varphi_1,...,\varphi_{n-1})$, generalizing the Toda system.
\end{enumerate}

Note that the degrees of freedom match: the matrix $\Phi_1$ being principal nilpotent implies that $d_A(\Phi)=0$ are effectively $n^2-n$ equations for the entries of $A_1$. Together with $t_k=\tr \Phi_1^{k-1}(\partial+A_1)$, we get $n^2-1$ equations for the entries of $A_1$. Since $F(A)+[\Phi\wedge\Phi^\dagger]$ is Hermitian, there are $(n-1)n/2$ degrees of freedom off the diagonal. Among them, $n-1$ give that the covector is $\mu$-holomorphic, and $(n-1)(n-2)/2$ serve to determine the entries in $\Phi_1$ which are not the $e^{\varphi_k}$. Finally, since $F(A)+[\Phi\wedge\Phi^\dagger]$ is traceless, there are $n-1$ independent terms on the diagonal, leading to the differential system for $(\varphi_1,...,\varphi_{n-1})$.

%In \cite[Theorem 7.8]{Fock-bundles}, the first point has been proven (not only in a local chart, but also globally over the surface). In addition, it is shown that the $\mu$-holomorphicity terms appear as the coefficients of the highest weight vectors of $F(A)+[\Phi\wedge\Phi^\dagger]$, for the action of a principal $\mathfrak{sl}_2$-triple (see Theorem 7.13 in \cite{Fock-bundles}).

%We expect the system of differential equations for $(\varphi_1,...,\varphi_{n-1})$ to have a unique global solution on $S$ for $(t_2,...,t_n)$ small enough (as it is the case for $n=2$, where we get the cosh-Gordon equation). This would then imply Conjecture \ref{nahc}.

\medskip
Let us describe a framework in which the generalized Toda system takes a nice form.
For this, we consider the family of connections $C(\l)=\l\Phi+d_A+\l^{-1}\Phi^{*_h}$ as a single connection with values in $\mathcal{L}(\mathfrak{sl}_n)$, the loop algebra of $\mathfrak{sl}_n$.

The \textbf{loop algebra} of $\mathfrak{sl}_n$ is defined by $\mc{L}(\mf{sl}_n) = \mf{sl}_n \otimes \C[\lambda,\lambda^{-1}]$, the space of Laurent polynomials with matrix coefficients.
An element of $\mc{L}(\mf{sl}_n)$ can equally be thought of as an infinite periodic matrix $(M_{i,j})_{i,j \in \mathbb{Z}}$ with $M_{i,j}=M_{i+n,j+n}$ and finite width (i.e. $M_{i,j}=0$ for all $\left| i+j \right|$ big enough).
The isomorphism is given as follows: to $\sum_{i=-N}^N N_i\l^i$ we associate $M_{i,j}=(N_{k_j-k_i})_{r_i,r_j}$ where $i=k_in+r_i$ and $j=k_jn+r_j$ are the Euclidean divisions of $i$ and $j$ by $n$ (so $0\leq r_i,r_j <n$), see also Figure \ref{affine-matrix}. One checks that the map is an algebra isomorphism.

\begin{figure}[h]
\centering
%\hspace*{-2cm}
%\includegraphics[width=5cm]{affine-matrix.png}
\begin{tikzpicture}[scale=1]

\draw (0,0)--(3,0);
\draw (0,1)--(3,1);
\draw (0,2)--(3,2);
\draw (0,3)--(3,3);
\draw (0,0)--(0,3);
\draw (1,0)--(1,3);
\draw (2,0)--(2,3);
\draw (3,0)--(3,3);

\draw (-0.3,1.5) node {$\hdots$};
\draw (3.3,1.5) node {$\hdots$};
\draw (1.5,-0.3) node {$\vdots$};
\draw (1.5,3.3) node {$\vdots$};
\draw (-0.3,3.3) node {$\ddots$};
\draw (3.3,-0.3) node {$\ddots$};

\draw [dashed, gray] (-0.3,2.3)--(2.3,-0.3);
\draw [dashed, gray] (0.7,3.3)--(3.3,0.7);

\draw [white, fill=white] (0.5,1.5) circle (0.3);
\draw [white, fill=white] (1.5,0.5) circle (0.3);
\draw [white, fill=white] (1.5,2.5) circle (0.3);
\draw [white, fill=white] (2.5,1.5) circle (0.3);

\draw (1.5,1.5) node {$N_0$};
\draw (2.5,0.5) node {$N_0$};
\draw (0.5,2.5) node {$N_0$};
\draw (1.5,2.5) node {$N_1$};
\draw (2.5,1.5) node {$N_1$};
\draw (2.5,2.5) node {$N_2$};
\draw (0.5,1.5) node {$N_{-1}$};
\draw (1.5,0.5) node {$N_{-1}$};
\draw (0.5,0.5) node {$N_{-2}$};

\draw [domain=-20:20] plot ({0.2+3.5*cos(\x)},{1.4+5.5*sin(\x)});
\draw [domain=159:199] plot ({2.7+3.5*cos(\x)},{1.4+5.5*sin(\x)});

\end{tikzpicture}

\caption{Affine matrix as infinite periodic matrix}
\label{affine-matrix}
\end{figure}

In the second viewpoint, a connection $C(\l)=\l\Phi+A+\l^{-1}\Phi^{*_h}$ with $\Phi$ strictly lower triangular (which can be achieved locally by Lemma \ref{phi1lower}) is precisely an \emph{infinite matrix with period $n$ and width $n$} (shown in Figure \ref{affine-matrix} by dashed lines). The ``tridiagonal'' property of the usual Toda system (see Equation \eqref{Eq:tridiagonal}) is replaced by the property ``width equal to periodicity''. 

The space of all 1-forms with values in the space of infinite matrices of period and width $n$, which we call $\mathcal{E}_n$, has a symplectic structure using the Atiyah--Bott form $ \omega=\int_S\tr \delta A\wedge\delta A$. Plugging in $C(\l)=\l\Phi+d_A+\l^{-1}\Phi^{*_h}$, we get
$$\omega_\l=\l\int_S\tr\delta\Phi\wedge\delta A+\int_S\tr(\delta\Phi\wedge\delta\Phi^{*_h}+\delta A\wedge\delta A)+\l^{-1}\int_S\tr\delta\Phi^{*_h}\wedge\delta A,$$
where we used that $\tr\delta\Phi\wedge\delta\Phi=0$ since the trace of two strictly lower diagonal matrices is zero.

Note that for trivial higher complex structures, we know from Lemma \ref{a1upper} that $A_1$ is upper triangular. Then in the viewpoint of infinite periodic matrices, the $(1,0)$-part $C_1(\l)$ is upper triangular ($\Phi_1$ is lower triangular but $\l\Phi_1$ is upper triangular in the infinite matrix) and the $(0,1)$-part $C_2(\l)$ is lower triangular, making the similarity with the Toda system \eqref{Eq:tridiagonal} more striking.

In order to further investigate this generalized Toda system, one has to determine the subspace of $\mathcal{E}_n$ determined by the $C(\l)$ arising from the data of a higher complex structure with covector $(t_k,\mu_k)_{2\leq k\leq n}$. 
The study of this generalized Toda system is subject of future research.

\appendix

\section{Proof of Lemma \ref{curvaturemodmod}}\label{appendix:C}

We prove here Lemma \ref{curvaturemodmod}, which states:
\begin{fakelemma}
Consider a family of $L$-parabolic connections of the form $C(\l)=\l\Phi+d_A+\l^{-1}\Psi$, where $\Phi$ comes from a higher complex structure and $d_A(\Phi)=0$. Then, the parabolic curvature $\xi_k(\l)$ satisfies
$$\xi_k(\l) \equiv (-\bar{\partial}\!+\!\bm\hat{\mu}_2\partial\!+\!k\partial\bm\hat{\mu}_k)\bm\hat{t}_k+\sum_{\ell=1}^{n-k}\left((\ell\!+\!k)\partial\bm\hat{\mu}_{\ell+2}+(\ell\!+\!1)\bm\hat{\mu}_{\ell+2}\del\right)\bm\hat{t}_{k+\ell} \mod \bm\hat{t}^2,\partial^2.$$
\end{fakelemma}

\begin{proof}
The statement is local, so we can restrict attention to an open chart $U\subset \S$. Recall the local basis $B=(s,\nabla s,...,\nabla^{n-1}s)$ where $s$ is a section of $\Gamma(L\!\mid_U)$. The curvature of the connection $C(\l)$ is given by $F(C(\l))=[\nabla,\bar\nabla]$, where $\nabla=\partial+\l\Phi_1+A_1+\l^{-1}\Psi_1$ and $\bar\nabla=\bar\partial+\l\Phi_2+A_2+\l^{-1}\Psi_2$. Since the curvature of an $L$-parabolic connection is concentrated in the first row, we get for $k=2,...,n-1$:
\begin{equation}\label{Eq:xik-aux}
\xi_k s = [\nabla,\bar\nabla]\nabla^{k-1}s = \nabla C_{k}-C_{k+1},
\end{equation}
where we denote by $C_k$ the $k$-th column of $\l\Phi_2+A_2+\l^{-1}\Psi_2$, the matrix of $\bar\nabla$. For $k=n$, we have 
$$\xi_n s=[\nabla,\bar\nabla]\nabla^{n-1}s=\nabla C_n-\bar\nabla(\bm\hat{t}_n s+...+\bm\hat{t}_2 \nabla^{n-2}s).$$ 
Applying $\nabla^{n-k}$ to Equation \eqref{Eq:xik-aux}, we get
$$\xi_k \nabla^{n-k}s+(n-k)\partial\xi_k \nabla^{n-k-1}s \equiv \nabla^{n-k}C_k-\nabla^{n-k+1}C_{k-1} \;\mod \partial^2.$$
Therefore by summation and putting $\xi_1=0$, we get
\begin{equation}\label{Eq:xi-eq-aux}
\sum_{k=2}^{n}(\xi_{k}+(n+1-k)\partial\xi_{k-1})\nabla^{n-k} s \equiv \nabla^{n-1}C_2-\bar\nabla(\bm\hat{t}_n s+...+\bm\hat{t}_2 \nabla^{n-2}s) \mod \partial^2.
\end{equation}
We will compute the right hand side in the basis $B$ and compare coefficients. This will give the lemma. 

%Note that $\partial\xi_k=0\mod \partial$ if the lemma is true. From Equation \eqref{Eq:xi-eq-aux}, we will prove the Lemma for $\xi_2$ by considering the coefficient of $\nabla^{n-2}s$. The coefficient of $\nabla^{n-3}s$ is $\xi_3+(n-2)\partial\xi_2$, which is $\xi_3$ modulo $\partial^2$ (since we will have proven the formula for $\xi_2$ before). The same argument can be iterated. Therefore, Equation \eqref{Eq:xi-eq-aux} can be simplified to
%\begin{align*}
%\xi_n s+\xi_{n-1}\nabla s+...+\xi_2 \nabla^{n-2}s &= \nabla^{n-1}C_2-\bar\nabla(\bm\hat{t}_n s+...+\bm\hat{t}_2 \nabla^{n-2}s)\\
%&= \nabla^{n-1}C_2-(\bar\partial\bm\hat{t}_n s+...+\bar\partial\bm\hat{t}_2\nabla^{n-2}s+\bm\hat{t}_nC_1+...+\bm\hat{t}_2C_{n-1}).
%\end{align*}

The right hand side of Equation \eqref{Eq:xi-eq-aux} equals $\nabla^{n-1}C_2-\sum_{k=2}^n(\bar\partial\bm\hat{t}_k\nabla^{n-k}s+\bm\hat{t}_kC_{n+1-k}).$
To evaluate $\bm\hat{t}_{k}C_{n+1-k}$ modulo $\bm\hat{t}^2$, we only have to compute $C_{k}$ modulo $\bm\hat{t}$. By definition, we have $C_1=\bm\hat{\mu}_1s+\bm\hat{\mu}_2\nabla s+...+\bm\hat{\mu}_{n}\nabla^{n-1}s$. A direct computation, using Equation \eqref{Eq:xik-aux} and the convention $\xi_1=0=\bm\hat{\mu}_0$, gives for $k=1,...,n-1$:
\begin{equation}\label{Eq:aux.3}
C_k\equiv \sum_{\ell=0}^{k-3} (\xi_{k-1-\ell}+\partial\xi_{k-2-\ell})\nabla^{\ell}s+\sum_{\ell=k-2}^{n-1}(\bm\hat{\mu}_{\ell-k+2}+(k-1)\partial\bm\hat{\mu}_{\ell-k+3})\nabla^{\ell}s \;\mod \bm\hat{t},\partial^2.
\end{equation}

To compute $\nabla^{n-1}C_2$, decompose $\nabla = \partial+A_1 = \partial+J+T$, where $J=\sum_{i=1}^{n-1}E_{i+1,i}$ and $T=\sum_{i=1}^{n-1}\bm\hat{t}_{n+1-i}E_{n,i}$, where $E_{i,j}$ denotes the standard basis of the space of matrices. The action of $J$ and $T$ on the basis $B$ is simple, namely $J\nabla^ks=\nabla^{k+1}s$ and $T\nabla^ks=0$ for $0\leq k\leq n-2$.
Using $C_2=\sum_{\ell=0}^{n-1}(\bm\hat{\mu}_\ell+\bm\hat{\mu}_n\bm\hat{t}_{n-\ell})\nabla^{\ell}s$, we then compute:
\begin{align}\label{eq:aux-10}
\nabla^{n-1}C_2 &\equiv A_1^{n-1}C_2+\sum_{k=0}^{n-2}A_1^k \partial(A_1^{n-2-k}C_2) \equiv \left(J^{n-1}+\textstyle\sum_{k=0}^{n-2}J^kTJ^{n-2-k}\right)C_2+\sum_{k=0}^{n-2} A_1^k \partial(A_1^{n-2-k}C_2)  \nonumber\\
&\equiv (\bm\hat{\mu}_n\bm\hat{t}_n+\partial\bm\hat{\mu}_1)\nabla^{n-1}s+\sum_{k=0}^{n-2}\sum_{\ell=k}^{n-1}\bm\hat{t}_{n+k-\ell}(\bm\hat{\mu}_{k+1}+\partial\bm\hat{\mu}_{k+2})\nabla^\ell s \nonumber\\
&\;\;\;\; + \sum_{k=0}^{n-2}A_1^k \partial((J+T)^{n-2-k}C_2) \mod \partial^2, \bm\hat{t}^2.
\end{align}
The last term needs more manipulation:
\begin{equation}\label{Eq:a1-k}
\sum_{k=0}^{n-2}A_1^k \partial((J+T)^{n-2-k}C_2) \equiv \sum_{k=0}^{n-2}A_1^k \partial\left(J^{n-2-k}C_2+\textstyle\sum_{m=0}^{n-3-k}J^mTJ^{n-3-k-m}C_2 \right) \mod\bm\hat{t}^2.
\end{equation}
The second part of Equation \eqref{Eq:a1-k} equals
\begin{align}\label{eq:aux-20}
\sum_{k=0}^{n-2}A_1^k \partial\sum_{m=0}^{n-3-k} J^mTJ^{n-3-k-m}C_2 &\equiv \sum_{k=0}^{n-2}J^k\partial \sum_{m=0}^{n-3-k}\sum_{\ell=m}^{n-1}\bm\hat{\mu}_{k+m+2}\bm\hat{t}_{n+m-\ell}\nabla^{\ell}s  \mod \bm\hat{t}^2\nonumber\\
&=\sum_{k=0}^{n-3}(k+1)\sum_{\ell=k}^{n-1}\partial(\bm\hat{\mu}_{k+2}\bm\hat{t}_{n+k-\ell})\nabla^\ell s,
\end{align}
where we reindexed by $k\mapsto m+k$.
The first part of Equation \eqref{Eq:a1-k} gives
\begin{align}\label{eq:aux-30}
\sum_{k=0}^{n-2}A_1^k\partial(J^{n-2-k}C_2) &\equiv \sum_{k=0}^{n-2} \sum_{\ell=n-2-k}^{n-1}A_1^k(\partial\bm\hat{\mu}_{\ell+k+2-n}+\partial(\bm\hat{\mu}_n\bm\hat{t}_{2n-\ell-k-2}))\nabla^{\ell}s \nonumber\\
&\equiv \sum_{k=0}^{n-2}\sum_{\ell=n-2-k}^{n-1}\left(J^k+\textstyle\sum_{m=0}^{k-1}J^mTJ^{k-1-m}\right)(\partial\bm\hat{\mu}_{\ell+k+2-n}+\partial(\bm\hat{\mu}_n\bm\hat{t}_{2n-\ell-k-2}))\nabla^{\ell}s \nonumber\\
&\equiv (n-1)\partial(\bm\hat{\mu}_n\bm\hat{t}_n)\nabla^{n-2}s+(n-1)(\partial(\bm\hat{\mu}_{n-1}\bm\hat{t}_n)+\partial\bm\hat{\mu}_1)\nabla^{n-1}s \nonumber\\
&\;\;\; + \sum_{k=0}^{n-3}(n-2-k)\sum_{\ell=k}^{n-1}\bm\hat{t}_{n+k-\ell}\partial\bm\hat{\mu}_{k+2} \nabla^\ell s \mod \partial^2, \bm\hat{t}^2.
\end{align}
The conclusion of Equations \eqref{eq:aux-10},\eqref{eq:aux-20},\eqref{eq:aux-30} is that for $0\leq \ell\leq n-3$, the coefficient in front of $\nabla^{\ell}s$ of $\nabla^{n-1}C_2$, modulo $\partial^2$ and $\bm\hat{t}^2$, is given by
\begin{align}\label{eq:aux-40}
&\sum_{k=0}^\ell (\bm\hat{t}_{n+k-\ell}(\bm\hat{\mu}_{k+1}+\partial\bm\hat{\mu}_{k+2})+(k+1)\partial(\bm\hat{\mu}_{k+2}\bm\hat{t}_{n+k-\ell})+(n-2-k)\bm\hat{t}_{n+k-\ell}\partial\bm\hat{\mu}_{k+2}) \nonumber\\
&= \sum_{k=0}^\ell (\bm\hat{t}_{n+k-\ell}\bm\hat{\mu}_{k+1}+n\bm\hat{t}_{n+k-\ell}\partial\bm\hat{\mu}_{k+2}+(k+1)\bm\hat{\mu}_{k+2}\partial\bm\hat{t}_{n+k-\ell}).
\end{align}
For $\ell=n-2$, we get the same expression since the term $(n-1)\partial(\bm\hat{\mu}_n\bm\hat{t}_n)$ of Equation \eqref{eq:aux-30} corresponds to the missing term for $k=n-2$ in Equation \eqref{eq:aux-20}.

Finally, we can combine Equations \eqref{Eq:xi-eq-aux}, \eqref{Eq:aux.3}, \eqref{eq:aux-40} and compare coefficients in front of $\nabla^{n-\ell}s$ to get:
\begin{align}\label{Eq:aux-50}
\xi_\ell+(n+1-\ell)\partial\xi_{\ell-1} &\equiv -\bar\partial\bm\hat{t}_\ell -\sum_{k=1}^{n-1}\bm\hat{t}_{n+1-k}(\xi_{k+\ell-1-n}+\partial\xi_{k+\ell-2-n}+\bm\hat{\mu}_{n-\ell-k+2}+(k-1)\partial\bm\hat{\mu}_{n-\ell-k+3})  \nonumber\\
& \;\;\; +\sum_{k=0}^{n-\ell}(\bm\hat{t}_{k+\ell}\bm\hat{\mu}_{k+1}+n\bm\hat{t}_{k+\ell}\partial\bm\hat{\mu}_{k+2}+(k+1)\bm\hat{\mu}_{k+2}\partial\bm\hat{t}_{k+\ell})\mod \partial^2,\bm\hat{t}^2,
\end{align}
where we used the notation $\xi_k=\bm\hat{\mu}_{k-1}=0$ for all $k\leq 1$.
We see that apart from the initial $\xi_\ell$, only terms $\xi_{\ell'}$ with $\ell'<\ell$ appear and those terms come with a derivative or a factor of $\bm\hat{t}$. 

We finish the proof using strong induction on $\ell$. The case of $\xi_2$ is directly given by Equation \eqref{Eq:aux-50}. Then, to prove the proposition for $\xi_\ell$, we can assume the formula for $\xi_{\ell'}$ with $\ell'<\ell$. In particular, we see that $\xi_{\ell'}\equiv 0 \mod \bm\hat{t}$ and also $\xi_{\ell'}\equiv 0 \mod\partial$. Therefore, we can remove all $\xi_{\ell'}$ with $\ell'<\ell$ from Equation \eqref{Eq:aux-50}. The remaining part is
$$\xi_\ell\equiv -\bar\partial\bm\hat{t}_\ell+\sum_{k=0}^{n-\ell}((\ell+k)\bm\hat{t}_{k+\ell}\partial\bm\hat{\mu}_{k+2}+(k+1)\bm\hat{\mu}_{k+2}\partial\bm\hat{t}_{k+\ell}) \mod \partial^2,\bm\hat{t}^2,$$
which is the formula we want to prove.
\end{proof}

\end{document}